\documentclass[reqno]{amsart}

\usepackage{amssymb}
\usepackage{amsmath}
\usepackage[mathscr]{euscript}

\makeatletter
\@addtoreset{equation}{section}
\makeatother

\renewcommand\thefigure{\thesection.\@arabic\c@figure}
\renewcommand\thetable{\thesection.\@arabic\c@table}

\newtheorem{theorem}{Theorem}[section]
\newtheorem{lemma}[theorem]{Lemma}
\newtheorem{proposition}[theorem]{Proposition}

\newtheorem{definition}[theorem]{Definition}
\newtheorem{remark}[theorem]{Remark}

\newcommand{\mc}[1]{{\mathcal #1}}
\newcommand{\mf}[1]{{\mathfrak #1}}
\newcommand{\mb}[1]{{\mathbf #1}}
\newcommand{\bb}[1]{{\mathbb #1}}

\newcommand{\<}{\langle}
\renewcommand{\>}{\rangle}
\renewcommand{\Cap}{{\rm cap}}

\newcommand{\oD}{\overline{D}}
\newcommand{\oE}{\overline{E}}
\newcommand{\oL}{\overline{L}}
\newcommand{\oM}{\overline{M}}
\newcommand{\oS}{\overline{S}}
\newcommand{\oV}{\overline{V}}
\newcommand{\oX}{\overline{X}}

\newcommand{\omu}{\overline{\mu}}

\newcommand{\omE}{\overline{\mc E}}

\newcommand{\oc}{\overline{c}}
\newcommand{\op}{\overline{p}}
\newcommand{\Or}{\overline{r}}

\newcommand{\ocap}{\overline{\Cap}}

\newcommand{\dv}{\text{\rm div }}

\begin{document}

\title[Dirichlet principle for non reversible Markov chains]
  {A Dirichlet principle for non reversible Markov chains
    and some recurrence theorems}

\author{A. Gaudilli\`ere, C. Landim}

\address{\noindent  Universit\'e de Provence, CNRS, 
  39, rue F. Joliot Curie, 13013 Marseille, France.
  \newline e-mail: \rm \texttt{gaudilli@cmi.univ-mrs.fr}
}

\address{\noindent IMPA, Estrada Dona Castorina 110, CEP 22460 Rio de
  Janeiro, Brasil and CNRS UMR 6085, Universit\'e de Rouen, Avenue de
  l'Universit\'e, BP.12, Technop\^ole du Madril\-let, F76801
  Saint-\'Etienne-du-Rouvray, France.  \newline e-mail: \rm
  \texttt{landim@impa.br} }


\begin{abstract}
  We extend the Dirichlet principle to non-reversible Markov processes
  on countable state spaces. We present two variational formulas for
  the solution of the Poisson equation or, equivalently, for the
  capacity between two disjoint sets.  As an application we prove a
  some recurrence theorems. In particular, we show the recurrence of
  two-dimensional cycle random walks under a second moment condition
  on the winding numbers.
\end{abstract}

\maketitle
\section{Introduction}

Since Kakutani work \cite{ka}, probability theory has not only proven
to be a powerful tool inside potential theory, but potential theory
has also given deep insight into the study of Markov processes. For
example, the Dirichlet and the Thomson principles, which express
escape probabilities as infima and suprema, respectively, give
efficient recurrence and transience criteria. One can use the
Dirichlet principle to prove the recurrence of random walks in random
conductances in dimension one and two, and the Thomson principle to
prove transience in dimension larger than or equal to three \cite{so,
  wo}.

Recently, potential theory and the Dirichlet principle played an
important role in the proof of almost sure convergence of Dirichlet
functions in transient Markov processes \cite{alp1, cfs1}, and in the
proof of the recurrence of a simple random walk on the trace of
transient Markov processes \cite{bgl1}. In a completely different
context, the Dirichlet principle has been a basic tool in the
investigation of metastability of reversible Markov processes
(cf. \cite{bhs1, bl3} and references therein).

Most applications of potential theory to Markov processes, as the ones
cited above, are however restricted to {\em reversible} processes due
to the lack of variational formulas for the effective resistance
between two sets in non-reversible processes. We fill this gap here,
presenting a Dirichlet principle for general Markov processes in
discrete state spaces.

To illustrate the interest of the Dirichlet principle, we present some
direct implications of this result. In Lemma \ref{ls01}, we extend to
non-reversible transient Markov processes a well known pointwise
estimate of a function in terms of its Dirichlet form and the Green
function.  In the last section, we state some recurrence theorems for
non-reversible processes.  In particular, we show that the recurrence
property of Durrett multidimensional generalization \cite{du} of Sinai
random walk relies in fact on the scale invariance properties of a
stationary measure and not on the reversibility of the process.  We
also give a sufficient second moment condition for the recurrence of
two-dimensional cyclic random walks considered before in \cite{ko1,
  ma, dk1, klo}.

In a completely different direction, relying on the Dirichlet
principle presented in this article we prove in \cite{gl3} the
metastable behavior of the condensate in supercritical asymmetric zero
range processes, extending to the non-reversible case the results
proved in \cite{bl3}.

\section{Notation and Main results}
\label{sec1}

Consider an irreducible Markov process $\{X_t : t\ge 0\}$ on a
countable state space $E$ with generator $L$. Denote by $\lambda(x)$,
$x\in E$, the holding rates, by $p(x,y)$, $x\not=y\in E$, the jump
probabilities, and by $r(x,y) = \lambda (x) p(x,y)$ the jump rates. In
particular, for every function $f:E\to \bb R$ with finite support,
\begin{equation}
\label{l02}
(Lf)(x) \;=\; \sum_{y\in E} r(x,y) [f(y)-f(x)]\;, \quad x\in E\;.
\end{equation}
Note that $Lf$ is well defined for a bounded function $f$. 

Assume that the Markov process $\{X_t : t\ge 0\}$ admits a stationary
state $\mu$. Let $L^2(\mu)$ be the space of square summable functions
$f:E \to \bb R$ endowed with the scalar product defined by
\begin{equation*}
\<f, g\>_\mu \;=\; \sum_{x\in E} \mu(x)\, f(x)\, g(x)\;.
\end{equation*}
Denote by the same symbol $L$ the generator acting on a domain of
$L^2(\mu)$, and by $D(f)$ the Dirichlet form or energy of a function
$f:E\to\bb R$ :
\begin{equation*}
D(f) \;=\; \frac 12 \sum_{x,y\in E} \mu(x)\,
r(x,y) \, [f(y)-f(x)]^2\;
\end{equation*}
so that for $f$ in the domain 
of the generator we have $D(f) = \<f, (-L)f\>_\mu$.

For each $x\in E$, denote by $\bb P_x$ the probability measure on the
path space $D(\bb R_+, E)$ of right continuous trajectories with left
limits induced by the Markov process $X_t$ starting from
$x$. Expectation with respect to $\bb P_x$ is denoted by $\bb E_x$.

Denote by $\{X^*_t : t\ge 0\}$ the stationary Markov process $X_t$
reversed in time. We shall refer to $X^*_t$ as the adjoint or the time
reversed process. It is well known that $X^*_t$ is a Markov process on
$E$ whose generator $L^*$ is the adjoint of $L$ in $L^2(\mu)$.  The
jump rates $r^*(x,y)$, $x\not=y\in E$, of the adjoint process satisfy
the balanced equations
\begin{equation}
\label{25}
\mu(x) \, r(x,y) \;=\; \mu(y) \, r^*(y,x)\;.
\end{equation}
Denote by $\lambda^*(x)=\lambda(x)$, $x\in E$, $p^*(x,y)$, $x\not=y\in
E$, the holding rates and the jump probabilities of the time reversed
process $X^*_t$.

As above, for each $x\in E$, denote by $\bb P^*_x$ the probability
measure on the path space $D(\bb R_+, E)$ induced by the Markov
process $X^*_t$ starting from $x$. Expectation with respect to $\bb
P^*_x$ is denoted by $\bb E^*_x$.

For a subset $A$ of $E$, denote by $T_A$ (resp. $T^+_A$) the hitting
(resp. return) time of a set $A$:
\begin{equation*}
\begin{split}
& T_A \,:=\, \inf \big\{ s > 0 : X_s \in A \big\}\;, \\
& \quad T^+_A \,:=\, \inf \{ t>0 : X_t\in A, X_s\not=X_0
\;\;\textrm{for some $0< s < t$}\}\;.
\end{split}
\end{equation*}
When the set $A$ is a singleton $\{a\}$, we denote $T_{\{a\}}$,
$T^+_{\{a\}}$ by $T_a$, $T^+_{a}$, respectively.
We set for every $x$ in $E$, $M(x) = \mu(x) \lambda(x)$.

\begin{definition}
\label{s06}
For two disjoint subsets $A$, $B$ of $E$, the capacity between $A$ and
$B$ is defined as
\begin{equation}
\label{27}
\Cap (A,B)\,=\, \sum_{x\in A} M(x) \, \bb P_x \big[ T_A^+ > T_B^+\big]\; .
\end{equation}
\end{definition}

Clearly, the sum may be infinite. We prove below in Lemma \ref{s04}
that the capacity is symmetric: $\Cap (A,B) = \Cap (B,A)$.

We may also express the capacity in terms of the distribution of the
adjoint process.  By \eqref{25}, for any sequence $x_0, x_1, \dots,
x_n$ such that $p(x_i, x_{i+1})>0$, $0\le i<n$, $M(x_0) \prod_{0\le
  i<n} p(x_i, x_{i+1}) = M(x_n) \prod_{0\le i<n} p^*(x_{i+1},x_i)$. It
follows from this observation that for any $a\in A$, $b\in B$, $M(a)
\bb P_a[ T_B< T^+_A, T_B=T_b] = M(b) \bb P^*_b[ T_A< T^+_B,
T_A=T_a]$. Hence,
\begin{equation*}
\Cap (A,B) \;=\; \sum_{a\in A} M(a) \bb P_a[ T_B< T^+_A] \;=\;
\sum_{a\in A} \sum_{b\in B} M(a) \bb P_a[ T_B< T^+_A, T_B=T_b]\;,
\end{equation*}
so that
\begin{equation}
\label{37}
\Cap (A,B) \;=\; \sum_{b\in B} M(b) \, \bb P^*_b[ T^+_A< T^+_B]\;
\;=\; \Cap^*(B,A).
\end{equation}

As in the reversible case, the capacity is a monotone function in each
of its coordinates:

\begin{lemma}
\label{s01}
Fix two disjoint subsets $A$, $B$ of $E$. Consider two sets $A'$, $B'$
such that $A\subset A' \subset B^c$ and $B\subset B' \subset
A^c$. Then,
\begin{equation*}
\Cap (A,B) \;\le\; \Cap (A,B')\;, \quad \Cap (A,B) \;\le\; \Cap (A',B) \;.
\end{equation*}
\end{lemma}

\begin{proof}
The first claim follows from the original definition 
and the second one from equation \eqref{37}. 
\end{proof}

For two disjoint subsets $A$, $B$ of $E$, let $V_{A,B}$, $V^*_{A,B}:
E\to [0,1]$ be the equilibrium potentials defined by
\begin{equation}
\label{06}
V_{A,B}(x) \;=\; \bb P_x[T_A < T_B]\;, \quad
V^*_{A,B}(x) \;=\; \bb P^*_x[T_A < T_B] \;.
\end{equation} 

When the set $B^c$ is finite, the equilibrium potential
$V_{A,B}$ has a finite support and belongs therefore to the domain of
the generator. Moreover, in this case, $V_{A,B}$ is the unique
solution of the elliptic equation
\begin{equation*}
\left\{
\begin{array}{l}
(LV)(z) \;=\; 0\; , \quad z\in E\setminus (A\cup B) \;, \\
V(x) =1\;, \quad x\in A\; , \\
V(y) =0\; , \quad y\in B\;.  
\end{array}
\right.
\end{equation*}
Furthermore, since by the Markov property, $-(L V_{A,B})(x) =
\lambda(x) \bb P_x[T_B < T_A^+]$, $x\in A$,
\begin{equation}
\label{l01}
\Cap (A,B)\,=\, \< V_{A,B} \,,\, (-L) \, V_{A,B} \>_\mu 
\;=\; D(V_{A,B})\; .
\end{equation}

This identity does not hold in general, since the scalar product is
not well defined if the set $B^c$ is not finite. However, following
\cite{bl2}, if the process $X_t$ is positive recurrent and the measure
$M(x) = \mu(x)\lambda(x)$ is finite, one can show that this formula
for the capacity holds.

\begin{lemma}
\label{s04}
For any disjoints subsets $A$ and $B$ of $E$,
\begin{equation*}
\Cap(A, B) \;=\; \Cap(B, A) \;.
\end{equation*}
Moreover, if $\{K_n \,|\, n\ge 1\}$ is an increasing sequence of
finite sets such that $E = \cup_{n\geq 1} K_n$, then
\begin{equation*}
\Cap(A, B) \;=\; \lim_{m\rightarrow +\infty} \lim_{n\rightarrow
  +\infty} \Cap(A_m, B_n) \;,
\end{equation*}
where $A_m = A\cap K_m$, $B_n = B\cup K_n^c$.
\end{lemma}

\begin{proof}
Assume first that $B^c$ is finite. In this case by \eqref{l01},
\begin{equation*}
\Cap(A,B) \;=\; D(V_{A,B}) \;=\; D(1 - V_{B, A}) \;=\; D(V_{B, A}) \;.
\end{equation*}
Since $\sum_{x\in B, y\in B^c} \mu(x)r(x,y)$ and $\sum_{y\in B, x\in
  B^c} \mu(x)r(x,y)$ are finite, and since $V_{B,A}$ is equal to $1$
on $B$, we may write $D(V_{B, A})$ as
\begin{equation}
\label{l04}
\begin{split}
&\frac{1}{2}\sum_{x,y} \mu(x)r(x,y)V_{B,A}(x)(V_{B,A}(x) - V_{B, A}(y))\\
& \quad +\; \frac{1}{2}\sum_{x,y} \mu(y)r(y,x)V_{B,A}(y)
(V_{B,A}(y) - V_{B, A}(x))\\
& \qquad +\; \sum_{x,y} c_a(x,y)V_{B,A}(y)(V_{B,A}(y) - V_{B,
  A}(x))\;, 
\end{split}
\end{equation}
where $c_a(x,y) = (1/2)[\mu(x)r(x,y) - \mu(y)r(y,x)]$ and all these
sums are absolutely convergent since $V_{B,A}$ is bounded. 

The first two lines of the previous sum are equal. Since $V_{B,A}$ is
bounded, $(L V_{B,A})(x)$ is well defined by \eqref{l02} for each $x$
in $E$. Moreover, $(L V_{B,A})(x) =0$ for $x\in (A\cup B)^c$ and $-(L
V_{B,A})(x) = \lambda(x) \bb P_x[T_A < T_B^+]$, $x\in B$.  Therefore,
the sum of the first two lines is equal to
\begin{equation}
\label{l03}
\sum_x \mu(x) V_{B,A}(x) (-LV_{B,A})(x)
\;=\; \sum_{b\in B} M(b) \, \bb P_b[T_A^+ < T_B^+] 
\;=\; \Cap(B, A)\;.
\end{equation}

On the other hand, since $c_a(x,y) = - c_a(y,x)$ and since the sum
$\sum_{x,y} u_{x,y}$ may be written as $(1/2) \sum_{x,y} \{u_{x,y} +
u_{y,x}\}$, the last line in \eqref{l04} is equal to
\begin{eqnarray*}
&& \frac{1}{2} \sum_{x,y} c_a(x,y)(V^2_{B,A}(y) -V^2_{B,A}(x))\\
&&\quad =\; \frac{1}{2} \sum_{x,y\not\in B} c_a(x,y)(V^2_{B,A}(y) -V^2_{B,A}(x))
+ \sum_{x\not\in B}\sum_{y\in B} c_a(x,y)(1- V^2_{B,A}(x))\\
&&\qquad =\;  -\sum_{x\not\in B} \sum_{y\in E} c_a(x,y)V^2_{B,A}(x)
+ \sum_{x\not\in B}\sum_{y\in B} c_a(x,y) \;.
\end{eqnarray*}
As $c_a(x,y) = - c_a(y,x)$, $\sum_{x,y\not\in B}c_a(x,y)=0$. We may
therefore replace the sum over $B$ in the last term by a sum over
$E$. Since $\mu$ is a stationary state, $\sum_{y\in E} c_a(x,y) =0$
for all $x\in E$. This proves that the last line of the previous
displayed formula vanishes. In conclusion, when $B^c$ is finite,
\begin{equation*}
\Cap(A,B) \;=\; D(V_{B,A}) \;=\; \Cap(B, A)\;.
\end{equation*}

It remains to remove the assumption that $B^c$ is finite.  Let $\{K_n
\,|\, n\ge 1\}$ be an increasing sequence of finite sets such that $E
= \cup_{n\geq 1} K_n$. For each $m\leq n$, let $A_m = A\cap K_m$, $B_n
= B\cup K_n^c$ and note that $B^c_n$ is finite for each $n\ge
1$. Since each set $A_m$ is finite, by \eqref{27}, by \eqref{37} and
by Beppo Levi's theorem,
\begin{eqnarray*}
&& \lim_{m\rightarrow +\infty} \lim_{n\rightarrow +\infty} \Cap(A_m, B_n)
\;=\; \lim_{m\rightarrow +\infty} \Cap(A_m, B)
\;=\; \lim_{m\rightarrow +\infty} \Cap^*(B, A_m) \\
&& \quad 
\;=\; \lim_{m\rightarrow +\infty} \sum_{b\in B} M(b) \, 
\bb P_b^*[T_B^+ > T_{A_m}^+] \;=\; \Cap^*(B, A) \;=\; 
\Cap(A, B)\;. 
\end{eqnarray*}
Since $B_n^c$ is finite, by \eqref{37} and by the first part of the
proof, for any $m\leq n$, $\Cap(A_m, B_n) = \Cap(B_n, A_m) =
\Cap^*(A_m, B_n)$. Repeating the same computations we obtain that
\begin{equation*}
\lim_{m\rightarrow +\infty} \lim_{n\rightarrow +\infty} \Cap(A_m, B_n)
\;=\; \lim_{m\rightarrow +\infty} \lim_{n\rightarrow +\infty}
\Cap^*(A_m, B_n) \;=\; \Cap(B, A)\;.
\end{equation*}
This proves the lemma.
\end{proof}

Denote by $S$ (resp. $A$) the symmetric (resp. anti-symmetric) part of
the generator $L$ in $L^2(\mu)$: $S=(1/2)\{L+L^*\}$,
$A=(1/2)\{L-L^*\}$. The next result is proved in Section \ref{sec2}.

\begin{theorem}
\label{s10}
Fix two disjoint subsets $A$, $B$ of $E$, with $B^c$ finite. Then,
\begin{equation}
\label{28}
\Cap (A,B)\,=\, \inf_F \, \sup_H \Big\{ 2 \<  L^* F \,,\, H\>_{\mu}  \, - 
\, \< H , (- S) H\>_{\mu} \Big\} \;,
\end{equation}
where the supremum is carried over all functions $H:E\to \bb R$ which
are constant at $A$ and $B$, and where the infimum is carried over all
functions $F$ which are equal to $1$ at $A$ and $0$ at $B$. Moreover,
the function $F_{A,B}$ which solves the variational problem for the
capacity is equal to $(1/2) \{ V_{A,B} + V^*_{A,B}\}$, where
$V_{A,B}$, $V^*_{A,B}$ are the harmonic functions defined in
\eqref{06}.
\end{theorem}

In the reversible case, the supremum over $H$ in the statement of
Theorem \ref{s10} is easily shown to be equal to $\< (-L) F \,,\,
F\>_{\mu}$ and we recover the well known variational formula for the
capacity:
\begin{equation*}
\Cap (A,B)\,=\, \inf_F \,  \<  (-L) F \,,\, F\>_{\mu} \;,
\end{equation*}
where the infimum is carried over all functions $F$ which are equal to
$1$ at $A$ and $0$ at $B$.  

When the set $E$ is finite and the sets $A$ and $B$ are singletons,
$A=\{a\}$, $B=\{b\}$, the supremum over $H$ becomes a supremum over
all functions $H:E\to\bb R$. In this case,
\begin{equation}
\label{32}
\sup_H \Big\{ 2 \<  L^* F \,,\, H\>_{\mu}  \, - 
\, \< H , (- S) H\>_{\mu} \Big\} \;=\; \< L^* F \,,\, (- S)^{-1} L^* F
\>_{\mu} \;.
\end{equation}
Therefore, when the set $E$ is finite, since $L \, (-S)^{-1} L^* =
\{[(-L)^{-1}]^s\}^{-1}$ \cite[Section 2.5]{klo}, the formula for the
capacity between two singletons becomes
\begin{equation*}
\Cap (\{a\}, \{b\})\,=\, \inf_F \, \< F \,,\, L (- S)^{-1} L^* F
\>_{\mu} \;=\; \inf_F \, \<F\,,\, \{[(-L)^{-1}]^s\}^{-1}F\>_\mu
\;,
\end{equation*}
where the infimum is carried over all functions $F$ which are equal to
$1$ at $a$ and $0$ at $b$, and where $[(-L)^{-1}]^s$ stands for the
symmetric part of the operator $(-L)^{-1}$. 

In Lemma \ref{s03} below we express the right hand side of \eqref{32}
as an infimum over divergence free flows. We have therefore two
alternative formulas for the scalar product $\< L^* F \,,\, (- S)^{-1}
L^* F \>_{\mu}$, one expressed as a supremum over functions, and
another one written as an infimum over flows.

\subsection{Estimates on the capacity}

We compare in this subsection the capacity associated to the generator
$L$ with the symmetric capacities $\Cap^s$ associated to the generators
$S$. Let $\{X^s_t : t\ge 0\}$ be the Markov process on $E$ with
generator $S$. We shall refer to $X^s_t$ as the symmetric or
reversible version of the process $X_t$.  Denote by $\bb P^s_x$, $x\in
E$, the probability measure on the path space $D(\bb R_+, E)$ induced
by the Markov process $X^s_t$ starting from $x$. 

For two disjoint subsets $A$, $B$ of $E$, let $\Cap^s (A,B)$ be the
capacity between the sets $A$ and $B$ for the reversible process
$X^s_t$:
\begin{equation*}
\Cap^s (A,B) \,=\, \sum_{x\in A} M(x) \, \bb P^s_x 
\big[ T_A^+ > T_B^+\big]\; .
\end{equation*}

In the case where the set $B^c$ is finite,
\begin{equation*}
\Cap^s (A,B) \,=\, \< V^s_{A,B} \,,\, (-S) \, V^s_{A,B}\>_\mu\; ,
\end{equation*} 
where $V^s_{A,B}$ is the equilibrium potential: $V^s_{A,B}(x) \;=\;
\bb P^s_x[T_A < T_B]$. Moreover, since the generator $S$ is symmetric
in $L^2(\mu)$, it is well known that if $B^c$ is finite,
\begin{equation}
\label{29}
\Cap^s (A,B) \;=\; \inf_F \, \<  (-S) F \,,\, F \>_{\mu} 
\;=\; \inf_F \, \<  (-L) F \,,\, F \>_{\mu}\;,
\end{equation}
where the infimum is carried over all functions $F$ which are equal to
$1$ at $A$ and $0$ at $B$.  Taking $H=-F$ in the variational formula
\eqref{28} we obtain that
\begin{equation*}
\Cap (A,B)\,\ge\, \inf_F \<  (-L) F \,,\, F \>_{\mu} \;.
\end{equation*}

The next result follows from the previous observation, \eqref{29} and
Lemma \ref{s04}.

\begin{lemma}
\label{s15}
For two disjoint subsets $A$, $B$ of $E$, 
\begin{equation*}
\Cap^s (A,B)\,\le\, \Cap (A,B)\;.
\end{equation*}
\end{lemma}

Recall that a generator $L$ satisfies a sector condition with constant
$C_0$ if for every $f$, $g$ in the domain of the generator,
\begin{equation*}
\< Lf, g\>_\mu^2 \;\le\; C_0 \, \< (-L) f, f\>_\mu \, \< (-L)g ,
g\>_\mu\;. 
\end{equation*}
Next result, whose proof is presented at the end of Section
\ref{sec2}, shows that if the generator $L$ satisfies a sector
condition, we may estimate the capacity between two sets by the
capacity associated to the symmetric part of the generator.

\begin{lemma}
\label{s11}
Suppose that the generator $L$ satisfies a sector condition with
constant $C_0$. Then, for every pair of disjoint subsets $A$, $B$ of
$E$, 
\begin{equation*}
\Cap (A,B) \;\le\; C_0 \, \Cap^s (A,B)\;.
\end{equation*}
\end{lemma}

\subsection{Flows in finite state spaces}

We have seen that one can reduce capacity computations to the case
when $B^c$ is finite.  By identifying $B$ with a single point (this
will be rigorously done in the next section) we can then restrict
ourselves to the case of a finite space $E$.  We then assume in this
subsection that the $E$ is \emph{finite}. In this case the stationary
measure $\mu$ is unique up to multiplicative constants. Define the
(generally asymmetric) conductances
\begin{equation*}
c(x,y) \;=\; \mu(x) \, r(x,y)\;, \quad 
c^*(x,y) \;=\; \mu(x) \, r^*(x,y)\;, \quad 
x\not = y\in E\;.
\end{equation*}
Note that $c(x,y) = c^*(y,x)$. Let $c_s(x,y)$, $c_a(x,y)$, be the
symmetric and the asymmetric parts of the conductances:
\begin{equation*}
c_s(x,y) \;=\; (1/2) \{c(x,y) + c^*(x,y)\}\;, \quad
c_a(x,y) \;=\; (1/2) \{c(x,y) - c^*(x,y)\} \;, 
\end{equation*}
for $ x\not = y\in E$.  Clearly, $c_s(x,y) = c_s(y,x)$ and $c_a(x,y) =
- c_a(y,x)$. The symmetric conductances $c_s(x,y)$ are the
conductances of the reversible Markov process associated to the
generator $S$.

Denote by $\mc E$ the set of oriented edges or arcs of $E$: $\mc E=
\{(x,y)\in E\times E : c_s(x,y)>0\}$. For an oriented edge $e=(x,y)\in
\mc E$, let $e^-=x$ be the tail of the arc $e$ and let $e^+=y$ be its
head. We call {\sl flow} any anti-symmetric function $\varphi:\mc E\to
\bb R$.  Denote by $\mc F$ the set of flows endowed with the scalar
product
\begin{equation}
\label{42}
\<\varphi , \psi\> \;=\; \frac 12 \sum_{(x,y)\in \mc E}
\frac 1{c_s(x,y)}\, \varphi (x,y)\, \psi (x,y)\;,
\end{equation}
and let $\Vert \,\cdot\,\Vert$ be the norm associated to this scalar
product. 

Denote by $(\text{\rm div } \varphi)(x)$, $x\in E$, the {\sl divergence} of
the flow $\varphi$ at $x$:
\begin{equation*}
(\text{\rm div } \varphi) (x) \;=\; \sum_{y:(x,y)\in \mc E} 
\varphi (x,y) \;, \quad x\in E\;.
\end{equation*}
A flow $\varphi$ whose divergence vanishes at all sites, $({\rm div }
\varphi) (x)=0$ for all $x\in E$, is called a \emph{divergence free}
flow. An important example of such a divergence free flow
is $c_a$.

For a function $f:E\to \bb R$, let $\Psi_f (x,y) = c_s(x,y) [f(x) -
f(y)]$ be the gradient flow associated to $f$. Clearly, $\Psi_f$
belongs to $\mc F$ and 
\begin{equation}
\label{10}
\Vert \Psi_f \Vert^2 \; =\;
\< \Psi_f , \Psi_f\> \; =\;  \<(-L)f \,,\, f\>_\mu\;. 
\end{equation}
Let $\mc G = \{\Psi_f \,|\, f :E \to \bb R \}\subset\mc F$. We refer
to $\mc G$ as the set of \emph{gradient flows}.

A finite sequence $\gamma = (x_0, \dots, x_n=x_0)$ of sites in $E$
which starts and ends at the same site is called a \emph{cycle} if
$(x_i,x_{i+1})$ is an arc for each $0\le i\le n-1$ and if $x_i\not =
x_j$ for $0\le i < j < n$. An arc $e$ is said to belong to a cycle
$\gamma = (x_0, \dots, x_n=x_0)$ if $e=(x_i,x_{i+1})$ for some $0\le
i<n$.  We associate to a cycle $\gamma = (x_0, \dots, x_n=x_0)$ the
flow $\chi_\gamma:\mc E\to \bb R$ defined by
\begin{equation}
\label{40}
\chi_\gamma  
\;=\; \sum_{i=0}^{n-1} \{ \delta_{(x_i,x_{i+1})} -
\delta_{(x_{i+1},x_i)} \}\;.
\end{equation}
Denote by $\mc C$ the subspace of $\mc F$ spanned by flows associated
to cycles.

A flow $\varphi\in\mc C$ associated to a cycle has no divergence:
\begin{equation*}
(\text{div } \varphi)(x) \;=\; 0 \quad x\in E\;.
\end{equation*}
Also, $\varphi$ is a gradient flow if and only if $\varphi$ is
orthogonal to all cycle flows.  In other words, we have
\begin{equation}
\label{31}
\mc F = \mc G \oplus \mc C\;, \quad
\mc G \perp \mc C\;.
\end{equation}
In addition, a flow $\varphi$ that is orthogonal to all gradient flows
satisfies, $({\rm div} \varphi)(x_0) = 0$ for all $x_0$ in $E$,
because $({\rm div} \varphi)(x_0) = \<\Psi_f, \varphi\>$ for the
function $f$ defined by $f (x)= \delta_{x_0,x}$. This proves that $\mc
C$ is the set of all divergence free flows.

Inspired by the computation of the current through an arc $(x,y)$,
presented in \eqref{38} below, for a function $f:E\to \bb R$, denote
by $\Phi_f : \mc E\to \bb R$ the flow defined by
\begin{equation}
\label{30}
\Phi_f (x,y) \;=\; f(x) \, c(x,y) \;-\; f(y) \, c(y,x) \;.
\end{equation}
In Section \ref{sec3} we prove the following result.

\begin{theorem}
\label{s02}
For any disjoint and non-empty sets $A$, $B\subset E$,
\begin{equation*}
\Cap (A,B) \;=\; \inf_f \, \inf_\varphi \,
\Vert \Phi_f - \varphi \Vert^2 \;,
\end{equation*}
where the first infimum is carried over all functions $f:E\to \bb R$
which are equal to $1$ on the set $A$ and $0$ on the set $B$, and the
second infimum is carried over all flows $\varphi\in \mc F$ such that
\begin{equation*}
(\dv \varphi)(x) = 0 \;, \quad x\in (A\cup B)^c\;, \quad
\sum_{a\in A} (\dv \varphi)(a) = 0\;, \quad
\sum_{b\in B} (\dv \varphi)(b) = 0\;.
\end{equation*}
\end{theorem}

In the case where $A$ and $B$ are singletons, the second infimum is
carried over all divergence free flows. Hence, in the case of
singletons, the infimum corresponds to a projection over the space of
gradient flows.

In the last section of this article, when we shall estimate some
capacities among singletons, the divergence free flow $\varphi(x,y) =
c_a(x,y)$, $(x,y)\in\mc E$, will be used repeatedly to obtain upper
bounds.

\subsection{Transient Markov processes} 

Assume in this subsection that the irreducible Markov process $\{X_t
\,|\, t\ge 0\}$ is transient, and denote by $G(x,y)$ its Green
function:
\begin{equation*}
G(x,y) \;=\; \bb E_{x} \Big[ \int_0^\infty \mb 1\{ X_t = y\} \, dt
\Big]\; .
\end{equation*}
Define the capacity of a state $x\in E$, denoted by $\Cap(x)$, as
\begin{equation*}
\Cap (x) \;=\; M(x) \, \bb P_x \big [ T^+_x = \infty\big]\;.
\end{equation*}
Since $G(x,x)^{-1} = \lambda (x) \bb P_x  [ T^+_x = \infty]$, we have
that 
\begin{equation}
\label{l09}
\Cap (x) \;=\; \mu(x)\, \frac 1{G(x,x)} \;\cdot
\end{equation}

Fix a finitely supported function $f: E\to \bb R$ such that $f(x) \not
= 0$, and let $F(y)= f(y)/f(x)$ so that $F(x)=1$.  By Definition
\ref{s06}, if $\{A_n \,|\, n\ge 1\}$ is a sequence of increasing,
finite sets such that $E= \cup_{n\ge 1} A_n$, 
\begin{equation*}
\Cap (x) \;=\; \lim_{n\to\infty} \Cap (x, A^c_n)\;.
\end{equation*}
Since $A_n$ is finite and since $F$ is finitely supported, with
$F(x)=1$, by Theorem \ref{s10},
\begin{equation*}
\Cap (x) \;\le \; \lim_{n\to\infty} \, \sup_{H\in \mf B_n} 
\Big\{ 2 \<  L^* F \,,\, H\>_{\mu}  \, - 
\, \< H , (- S) H\>_{\mu} \Big\} \;,
\end{equation*}
where $\mf B_n$ is the set of functions $H:E\to\bb R$ which vanish at
$A^c_n$. As $F(\,\cdot\,)= f(\,\cdot\,)/f(x)$, replacing $H$ by
$H'(\,\cdot\,)= H(\,\cdot\,)/f(x)$, we obtain that
\begin{equation*}
\Cap (x) \;\le \; \frac 1{f(x)^2} \lim_{n\to\infty} \, \sup_{H\in \mf B_n} 
\Big\{ 2 \<  L^* f \,,\, H\>_{\mu}  \, - 
\, \< H , (- S) H\>_{\mu} \Big\} \;.
\end{equation*}
In view of \eqref{l09}, we have proved the following result, which
generalizes a well known estimate in the context of reversible Markov
processes \cite[Proposition 5.23]{klo}, \cite[Lemma 2.1]{alp1}.

\begin{lemma}
\label{ls01}
Let $f:E\to\bb R$ be a finitely supported function and let $\{A_n \,|\,
n\ge 1\}$ be a sequence of increasing, finite sets such that $E=
\cup_{n\ge 1} A_n$. Then, for every $x\in E$,
\begin{equation*}
\mu(x)\, f(x)^2 \;\le\; G(x,x)\, 
\lim_{n\to\infty} \, \sup_{H\in \mf B_n} 
\Big\{ 2 \<  L^* f \,,\, H\>_{\mu}  \, - 
\, \< H , (- S) H\>_{\mu} \Big\} \;,
\end{equation*}
where $\mf B_n$ is the set of functions $H:E\to\bb R$ which vanish at
$A^c_n$.
\end{lemma}

\section{Collapsed Chains and Proof of Theorem \ref{s10}}
\label{sec2}

We start this section by assuming that $E$ is \emph{finite} and that
$\mu$ is the unique stationary probability measure. In the case where
the sets $A$ and $B$ are singletons, the proof of Theorem \ref{s10}
takes the following form.

\begin{lemma}
\label{s05}
Fix a pair of points $a\not = b$ in a finite set $E$. Then,
\begin{equation}
\label{07}
\Cap (\{a\}, \{b\}) \;=\; \inf_f \, \<f\,,\, L \, (-S)^{-1} L^*f\>_\mu\;,
\end{equation}
where the infimum is carried over all function $f:E\to \bb R$ such
that $f(a)=1$, $f(b)=0$. Moreover, the function $f_{a,b}$ which solves
the variational problem \eqref{07} is unique and equal to $(1/2) \{
V_{a,b} + V^*_{a,b}\}$, where $V_{a,b}$, $V^*_{a,b}$ are the harmonic
functions defined in \eqref{06}.
\end{lemma}

\begin{proof}
The operator $L(-S)^{-1}L^*$ restricted to the space of mean zero
functions is symmetric, strictly positive and bounded because the
state space is finite. There exists, in particular, a unique
function $f_{a,b}$ which solves the variational problem
\eqref{07}. Moreover, as $(-S)^{-1}$ is also strictly positive on the
the space of mean zero functions, there exists a strictly positive
constant $C_0$ such that
\begin{equation}
\label{02}
\<f_{a,b}\,,\, L \, (-S)^{-1} L^*f_{a,b}\>_\mu \;\ge\; 
C_0 \<L^* f_{a,b}\,,\, L^*f_{a,b}\>_\mu \;>\; 0\;.
\end{equation}
The previous expression can not vanish due to the boundary conditions
of $f_{a,b}$.

Since $f_{a,b}$ solves the variational problem \eqref{07},
\begin{equation*}
(L \, (-S)^{-1} L^*f_{a,b})(x) \;=\; 0 \;, \quad x\not = a,b\;.
\end{equation*}
Let $W_{a,b} = S^{-1} L^*f_{a,b} + c_0$, where $c_0$ is a constant
chosen for $W_{a,b}$ to vanish at $b$: $W_{a,b}(b)=0$. Since $LW_{a,b}
=0$ on $E\setminus \{a,b\}$, $W_{a,b}$ is a multiple of the harmonic
function $V_{a,b}$ introduced in \eqref{06}: $W_{a,b} = \lambda
V_{a,b}$, where $\lambda = W_{a,b}(a)$.

We claim that $\lambda=1$ so that $W_{a,b} = V_{a,b}$. Indeed, since
$W_{a,b}$ is harmonic on $E\setminus \{a,b\}$ and $f_{a,b}(a)=1$,
$f_{a,b}(b)=0$,
\begin{equation*}
\<f_{a,b}\,,\, L \, (-S)^{-1} L^*f_{a,b}\>_\mu \;=\; 
\<f_{a,b}\,,\, (-L) \, W_{a,b}\>_\mu \;=\; - \mu(a) \, (L \,
W_{a,b})(a) \;.
\end{equation*}
On the other hand, since $W_{a,b} -c_0 =S^{-1} L^*f_{a,b}$ and $S
S^{-1}$ is the identity,
\begin{equation*}
\begin{split}
& \<f_{a,b}\,,\, L \, (-S)^{-1} L^*f_{a,b}\>_\mu \;=\; 
\<L^* f_{a,b}\,,\, (-S)^{-1} L^*f_{a,b}\>_\mu \;=\; 
\< W_{a,b}\,,\, (-S) W_{a,b}\>_\mu \\
& \qquad \;=\; \< W_{a,b}\,,\, (-L) W_{a,b}\>_\mu   
\;=\; - \mu(a) \, W_{a,b}(a)\, (L \, W_{a,b})(a) \;,
\end{split}
\end{equation*}
where the last identity follows from the fact that $W_{a,b}$ is
harmonic on $E\setminus \{a,b\}$ and that $W_{a,b}(b)=0$. By
\eqref{02} and the two previous displayed formulas, $W_{a,b} (a)=1$ so
that $W_{a,b} = V_{a,b}$. Hence, by the last displayed formula and
\eqref{27},
\begin{equation*}
\<f_{a,b}\,,\, L \, (-S)^{-1} L^*f_{a,b}\>_\mu \;=\; 
\< V_{a,b}\,,\, (-L) V_{a,b}\>_\mu \;=\; \Cap (\{a\}, \{b\})\; ,
\end{equation*}
which concludes the proof of the first assertion of the lemma.

Denote by $f_{a,b}$ a function which solves the variational problem
\eqref{07}. We claim that $f_{a,b} = (1/2) \{ V_{a,b} + V^*_{a,b}
\}$. Indeed, since $W_{a,b}= V_{a,b}$ and since $V_{a,b}$ is
$L$-harmonic on $E\setminus \{a,b\}$, on this set $(1/2) L^*V_{a,b} =
SV_{a,b} = L^* f_{a,b}$. Furthermore, as $V^*_{a,b}$ is $L^*$-harmonic
on $E\setminus \{a,b\}$, we have in fact that $(1/2) L^* \{ V_{a,b} +
V^*_{a,b} \} = L^* f_{a,b}$. Hence,
\begin{equation*}
\begin{split}
& L^* (1/2)  \{ V_{a,b} + V^*_{a,b} \} \; =\; L^* f_{a,b} \text{ on }
E\setminus \{a,b\} \\
&\quad \text{and $(1/2) \{ V_{a,b} + V^*_{a,b} \} \; =\;
f_{a,b}$ on $\{a,b\}$}\;.
\end{split}
\end{equation*}
It follows from these two identities that $f_{a,b} = (1/2) \{ V_{a,b}
+ V^*_{a,b} \}$.
\end{proof}

To extend the previous result to the case where the sets $A$ and $B$
are not singletons, we define a Markov chain in which a set is
collapsed to a single state. Fix a subset $A$ of $E$, and let $\oE_A=
[E \setminus A] \cup \{\mf d\}$, where $\mf d$ is an extra site added
to $E$ to represent the collapsed set $A$. Denote by $\{\oX^A_t : t\ge
0\}$ the chain obtained from $X_t$ by collapsing the set $A$ to a
singleton. This is the Markov process on $\oE_A$ with jump rates
$\Or_A(x,y)$, $x$, $y\in \oE_A$, given by
\begin{equation}
\label{12}
\begin{split}
& \Or_A(x,y) \;=\; r(x,y) \;, 
\quad \Or_A(x,\mf d) \;=\; \sum_{z\in A} r(x,z) \;, 
\quad x, y \in E \setminus A\;, \\
& \quad \Or_A(\mf d,x) \;=\; \frac 1{\mu(A)} \sum_{y\in A} \mu(y) \, 
r(y,x) \;, \quad x \in E \setminus A \;. 
\end{split}
\end{equation}
The collapsed chain $\{\oX^A_t : t\ge 0\}$ inherits the irreducibility
from the original chain.

Denote by $\omu_A$ the probability measure on $\oE_A$ given by
\begin{equation}
\label{16}
\omu_A (\mf d) \;=\; \mu(A)\;, \quad 
\omu_A(x) \;=\; \mu(x)\;, \quad x \in E \setminus A \;.
\end{equation}
Since
\begin{equation*}
\sum_{y\not\in A , z\in A} c(y,z) \;=\; \sum_{y\not\in A , z\in A} c(z,y)\;,
\end{equation*}
one checks that $\omu_A$ is a stationary state, and
therefore the unique invariant probability measure, for the collapsed
chain $\oX^A_t$.

We may extend the concept of collapsed chain to the case in which more
than one set is collapsed to a singleton. One can proceed recursively,
collapsing first a set $A$ to a point $a\not \in E$, obtaining a
Markov chain in $(E\setminus A)\cup \{a\}$, and then collapsing a set
$B\subset E$, $B\cap A=\varnothing$, to a point $b\not\in E
\cup\{a\}$, obtaining a new Markov chain in $[E\setminus (A\cup
B)]\cup \{a,b\}$. One checks that the final process is the
same if we first collapse $B$ and then $A$. The rate $\Or_{A, B}(a,b)$
at which the collapsed chain jumps from $a$ to $b$ is given by
\begin{equation}
\label{14}
\Or_{A, B}(a,b) \;=\; \frac 1{\mu(A)} \sum_{z\in A} \mu(z)
\sum_{x\in B} \, r(z,x)\;.
\end{equation}

Denote by $\oL_A$ the generator of the chain $\{\oX^A_t : t\ge 0\}$
and by $\oL_A^{\ *}$ the adjoint of $\oL_A$ in $L^2(\omu_A)$. Recall
that we represent by $\{X^*_t : t\ge 0\}$ the adjoint of the chain
$X_t$ and by $L^*$ its generator. Let $\{\overline{X^*}^A_t : t\ge
0\}$ be the chain obtained from $X^*_t$ by collapsing the set $A$ to a
singleton and by $\overline{L^*}_A$ the generator of this process.
We claim that
\begin{equation}
\label{15}
\oL_A^{\ *} \;=\; \overline{L^*}_A\;.
\end{equation}
To prove this claim, denote by $\Or_A^{\, *}(x,y)$ the rates of the adjoint
of $\{\oX^A_t : t\ge 0\}$:
\begin{equation*}
\Or_A^{\, *}(x,y) \;=\; \frac{\omu_A(y) \, \Or_A(y,x)} {\omu_A(x)}\;,
\quad x\,,\, y \in \oE_A \;.
\end{equation*}
Let $r^*(x,y)$, $x$, $y\in E$, be the jump rates of the adjoint
process and let $\overline{r^*}_A(x,y)$, $x$, $y\in \oE_A$, be the
jump rates of its collapsed version.

In view of the previous displayed formula and by \eqref{12},
\eqref{16}, for $x$, $y\in E\setminus A$,
\begin{equation*}
\Or_A^{\, *}(x,y) \;=\; \frac{\mu(y) \, r(y,x)} {\mu(x)} \;=\; r^*(y,x) \;=\;
\overline{r^*}_A(x,y) \;. 
\end{equation*}
Furthermore, for $y\in E\setminus A$, since $\omu_A(\mf d) = \mu(A)$, by
\eqref{12}, 
\begin{equation*}
\begin{split}
\Or_A^{\, *}(\mf d,y) \; & =\; \frac{\mu(y) \, \Or_A(y,\mf d)} {\mu(A)}
\;=\; \frac{\mu(y) \sum_{z\in A} r(y,z)} {\mu(A)} \\
\;& =\; \frac{\sum_{z\in A} \mu(z) \, r^*(z,y)} {\mu(A)} \;=\;
\overline{r^*}_A(\mf d ,y)\;.
\end{split}
\end{equation*}
Finally, for $x\in E\setminus A$, by analogous reasons,
\begin{equation*}
\begin{split}
\Or_A^{\, *}(x,\mf d) \; & =\; \frac{\mu(A) \, \Or_A(\mf d,x)} {\mu (x)}
\;=\; \frac{\sum_{z\in A} \mu(z) r(z,x)} {\mu (x)} \\
\; & =\; \sum_{z\in A} r^*(x,z) \;=\; \overline{r^*}_A (x,\mf d)\;,
\end{split}
\end{equation*}
which proves \eqref{15}. It follows from this result that
\begin{equation}
\label{19}
\oS_A \;=\; (1/2) \big\{ \oL_A + \oL_A^{\, *} \big\}\;,
\end{equation}
if $\oS_A$ stands for the generator $S= (1/2) (L+L^*)$ collapsed on the
set $A$. 

Fix two functions $f$, $g:\oE_A\to \bb R$. Let $F$, $G:E\to \bb R$ be
defined by $F(x) = f(x)$, $x\in E\setminus A$, $F(z) = f(\mf d)$,
$z\in A$, with a similar definition for $G$. We claim that
\begin{equation}
\label{17}
\< \oL_A f \,,\, g \>_{\omu_A} \;=\; \< L F \,,\, G \>_{\mu} \;.
\end{equation}
Conversely, if $F$, $G:E\to \bb R$ are two functions constant over
$A$, \eqref{17} holds if we define $f$, $g: \oE_A\to \bb R$ by $f(x) =
F(x)$, $x\in E\setminus A$, $f(\mf d) = F(z)$ for some $z\in A$.

Fix two functions $f$, $g: \oE_A\to \bb R$. By definition of $\oL_A$, 
\begin{equation*}
\< \oL_A f \,,\, g \>_{\omu_A} \;=\; \sum_{x, y \in \oE_A} \omu_A(x) \,
\Or_A(x,y) \, [f(y)-f(x)]\, g(x)\;. 
\end{equation*}
In view of \eqref{12}, \eqref{16}, this expression is equal to
\begin{equation*}
\begin{split}
& \sum_{x\in E\setminus A} \mu (x) \Big\{ \sum_{y\in E\setminus A} 
r(x,y) \, [f(y) - f(x)] + \sum_{z\in A} r(x,z) \, [f(\mf d) - f(x)] \Big\}
\, g(x) \\
& \qquad \;+\; \sum_{y\in E\setminus A} \sum_{z\in A} \mu(z)\, r (z,y) \,
[f(y)-f(\mf d)] \, g(\mf d) \;.
\end{split}
\end{equation*}
Since $F(x)=f(x)$ for $x\in E\setminus A$, and $F(y)=f(\mf d)$ for
$y\in A$, with similar identities with $G$, $g$ replacing $F$, $f$,
the last sum is equal to
\begin{equation*}
\begin{split}
& \sum_{x\in E\setminus A} \mu (x) \Big\{ \sum_{y\in E\setminus A} 
r(x,y) \, [F(y) - F(x)] + \sum_{z\in A} r(x,z) \, [F(z) - F(x)] \Big\}
\, G(x) \\
&\qquad \;+\; \sum_{z\in A}  \sum_{y\in E\setminus A} 
\mu(z) \, r(z,y) \, [F(y)-F(z)] \, G(z)\;.
\end{split}
\end{equation*}
Since $F$ is constant on $A$, we may add to this expression
\begin{equation*}
\sum_{x\in A}  \sum_{y\in A} \mu(x) \, r(x,y) \, [F(y)-F(x)] \, G(x)
\end{equation*}
to obtain that the last displayed expression is equal to $\<LF ,
G\>_\mu$, which concludes the proof of the first assertion of
\eqref{17}. The second statement is obtained following the computation
in reverse order. 

It follows from \eqref{15} and \eqref{17} that
\begin{equation}
\label{18}
\< \overline {L^*}_A f \,,\, g \>_{\omu_A} \;=\; \< L^* F \,,\, G \>_{\mu} \;,
\quad \< \oS_A f \,,\, g \>_{\omu_A} \;=\; \< S F \,,\, G \>_{\mu} \;.
\end{equation}

The next assertion establishes the relation between collapsed chains
and capacities.  Fix two disjoint subsets $A$ and $B$ of $E$. Let
$\overline{E}_{A,B} = [E \setminus (A\cup B)] \cup \{a,b\}$, where
$a\not = b$ are states which do not belong to $E$.  Denote by
$\{\overline{X}^{A, B}_t : t\ge 0\}$ the chain in which the sets $A$,
$B$ have been collapsed to the states $a$, $b$. Let $\Or_{A,B} (x,y)$,
$\op_{A,B} (x,y)$, and $\overline{\lambda}_{A,B} (x)$, $x$, $y\in
\overline{E}_{A,B}$, be the jump rates, the jump probabilities, and
the holding rates, respectively, of the chain $\overline{X}^{A, B}_t$,
and denote by $\omu_{A,B}$ its unique invariant probability measure.

Denote by $\overline{\Cap}_{A,B}$ the capacity associated to the
collapsed chain. We claim that
\begin{equation}
\label{11}
\overline{\Cap}_{A,B}(\{a\}, \{b\}) \;=\; \Cap (A,B)\;.
\end{equation}
Denote by $\overline{\bb P}^{A,B}_x$, $x\in \overline{E}_{A,B}$, the
probability measure on $D(\bb R_+, \overline{E}_{A,B})$ induced by the
collapsed chain $\overline{X}^{A, B}_t$ starting from $x$. By Definition
\ref{s06},
\begin{equation*}
\begin{split}
\overline{\Cap}_{A,B} (\{a\}, \{b\}) \; & =\; \overline{M}_{A,B} (a)\, 
\overline{\bb P}^{A,B}_a \big[\, T^+_a > T^+_b\, \big] \\
\;& =\; \overline{M}_{A,B}(a)\, \sum_{x\in \overline{E}_{A,B}}  
\overline{p}_{A,B} (a,x) 
\, \overline{\bb P}^{A,B}_x \big[\, T_a > T_b\, \big]\;,
\end{split}
\end{equation*}
where $\overline{M}_{A,B} (x) = \omu_{A,B} (x)
\overline{\lambda}_{A,B} (x)$.  Since $\oM_{A,B}(a)\, \op_{A,B} (a,x)
= \omu_{A,B}(a)\, \Or_{A,B} (a,x)$ and since $\op_{A,B} (a,a) =0$, by
the explicit expression \eqref{12} for the rates of the collapsed
chain, the previous expression is equal to
\begin{equation*}
\begin{split}
&\omu_{A,B} (a) \sum_{x\in E\setminus [A\cup B]} \frac 1{\mu (A)} 
\sum_{z\in  A}  \mu (z) \, r(z,x) \, \overline{\bb P}^{A,B}_x 
\big[\, T_a > T_b\, \big] \\
& \quad \;+\; 
\omu_{A,B} (a) \, \Or_{A,B}(a,b) \, \overline{\bb P}^{A,B}_b 
\big[\, T_a > T_b\, \big]\;. 
\end{split}
\end{equation*}
By construction, $\overline{\bb P}^{A,B}_x [\, T_a > T_b\, ] = \bb P_x
[\, T_A > T_B \, ]$ for $x\in E \setminus [A\cup B]$, and
$\overline{\bb P}^{A,B}_b [\, T_a > T_b\,] = 1 = \bb P_x [\, T_A > T_B
\, ]$, $x\in B$. Hence, as $\omu_{A,B} (a) = \mu (A)$, by \eqref{14}
the last sum is equal to
\begin{equation*}
\sum_{x\in E\setminus A} \sum_{z\in A}  \mu (z) \, r(z,x) \,
\bb P_x \big[\, T_A > T_B \, \big] \;.
\end{equation*}
Since $\bb P_x [\, T_A > T_B \, ] = 0$, $x\in A$, and since $\mu (z)
\, r(z,x) = M(z) p(z,x)$, this expression is equal to
\begin{equation*}
\;=\; \sum_{z\in A} M(z) \, \bb P_z \big[\, T^+_A > T^+_B\, \big]
\;=\; \Cap (A,B)\;,
\end{equation*}
which concludes the proof of claim \eqref{11}.

\begin{lemma}
\label{s14}
Fix two disjoint subsets $A$, $B$ of a finite set $E$. Then,
\begin{equation*}
\Cap (A,B)\,=\, \inf_F \, \sup_H \Big\{ 2 \<  L^* F \,,\, H\>_{\mu}  \, - 
\, \< H , (- S) H\>_{\mu} \Big\} \;.
\end{equation*}
where the supremum is carried over all functions $H:E\to \bb R$ which
are constant at $A$ and $B$, and where the infimum is carried over all
functions $F$ which are equal to $1$ at $A$ and $0$ at $B$. Moreover,
the function $F_{A,B}$ which solves the variational problem for the
capacity is equal to $(1/2) \{ V_{A,B} + V^*_{A,B}\}$, where
$V_{A,B}$, $V^*_{A,B}$ are the harmonic functions defined in
\eqref{06}.
\end{lemma}

\begin{proof}
Fix two disjoint subsets $A$, $B$ of $E$. By Lemma \ref{s05} and
identity \eqref{11},
\begin{equation}
\label{13}
\Cap (A,B) \;=\; \inf_f \, \< \oL_{A, B}^{\, *} f\,,\, \, 
(- \mc S)^{-1} \oL_{A, B}^{\, *} f\>_{\omu_{A, B}}\;,
\end{equation}
where $\oL_{A, B}$ is the generator of the chain $\{\overline{X}^{A,
  B}_t : t\ge 0\}$ introduced right after \eqref{18}, $\mc S$ is the
symmetric part of $\oL_{A, B}$, $\mc S = (1/2) (\oL_{A, B} + \oL_{A,
  B}^{\, *})$, and where the infimum is carried over all function
$f:\oE_{A, B}\to \bb R$ such that $f(a)=1$, $f(b)=0$.

By the variational formula for the norm induced by the operator $(-
\mc S)^{-1}$, the previous expression is equal to
\begin{equation*}
\inf_f \, \sup_h \Big\{ 2 \< \oL_{A, B}^{\, *}  f\,,\, h\>_
{\omu_{A, B}}  \, - \,
\< h, (-\mc S) h \>_{\omu_{A, B}} \Big\}\;,
\end{equation*}
where the supremum is carried over all functions $h:\oE_{A, B}\to \bb
R$. By \eqref{15}, $\oL_{A, B}^{\, *} = \overline{L^*}_{A, B}$, and by
\eqref{19}, $\mc S = \oS_{A, B}$, where $\oS_{A, B}$ is the generator
$S$ collapsed at $A$ and $B$. Hence, the previous displayed equation is
equal to
\begin{equation*}
\inf_f \, \sup_h \Big\{ 2 \<  \overline{L^*}_{A, B} f\,,\, h\>_{
  \omu_{A, B}}  \, - \,
\< h, (- \oS_{A, B}) h \>_{\omu_{A, B}} \Big\}\;.
\end{equation*}
Let $F$, $H:E\to \bb R$ be defined by $F(x) = f(x)$, $x\in E\setminus
(A\cup B)$, $F(z) = f(a)$, $z\in A$, $F(y) = f(b)$, $y\in B$, with a
similar definition for $H$. By \eqref{17}, \eqref{18}, the last
variational problem can be rewritten as
\begin{equation}
\label{20}
\inf_F \, \sup_H \Big\{ 2 \<  L^* F \,,\, H\>_{\mu}  \, - 
\, \< H , (- S) H\>_{\mu} \Big\} \;.
\end{equation}
where the supremum is carried over all functions $H:E\to \bb R$ which
are constant at $A$ and $B$, and where the infimum is carried over all
functions $F$ which are equal to $1$ at $A$ and $0$ at $B$. This
proves the first assertion of the lemma. 

To prove the second assertion of the lemma, recall from Lemma
\ref{s05} that
\begin{equation*}
\Cap (A,B) \;=\; \, \< \oL_{A, B}^{\, *} f_{a,b}\,,\, \, 
(- \mc S)^{-1} \oL_{A, B}^{\, *} f_{a,b}\>_{\omu_{A, B}}\;,
\end{equation*}
where $f_{a,b} = (1/2)\{\oV_{a,b} + \oV^*_{a,b}\}$ and $\oV_{a,b}$,
$\oV^*_{a,b}$ are the harmonic functions for the collapsed process. By
the first part of the proof, the right hand side is equal to
\begin{equation*}
\sup_H \Big\{ 2 \<  L^* F_{A,B} \,,\, H\>_{\mu}  \, - 
\, \< H , (- S) H\>_{\mu} \Big\} \;,
\end{equation*}
where the supremum is carried over all functions $H:E\to \bb R$ which
are constant at $A$ and $B$, and where $F_{A,B}(x) = f_{a,b}(x)$,
$x\in E\setminus (A\cup B)$, $F_{A,B}(z) = 1$, $z\in A$, $F_{A,B}(y) =
0$, $y\in B$. As we have already seen, by construction of the
collapsed process, for $x\in E\setminus (A\cup B)$,
\begin{equation*}
\oV_{a,b} (x) \;=\; \overline{\bb P}^{A,B}_x [\, T_a < T_b\, ] \;=\; 
\bb P_x [\, T_A < T_B \, ] \;=\; V_{A,B}(x)\;,
\end{equation*}
with a similar identity for $\oV_{a,b}^*$. In conclusion,
\begin{equation*}
\Cap (A,B) \;=\; \sup_H \Big\{ 2 \<  L^* F_{A,B} \,,\, H\>_{\mu}  \, - 
\, \< H , (- S) H\>_{\mu} \Big\} \;,
\end{equation*}
where $F_{A,B} = (1/2)\{V_{A,B} + V_{A,B}^*\}$, concluding the proof
of the lemma.  
\end{proof}

\begin{remark}
\label{s19}
The expression inside braces in the displayed formula of
Lemma \ref{s14} does not change if $H$ is replaced by $H+c$, where $c$
is a constant. We may therefore restrict the supremum to functions $H$
which vanish at $B$.  
\end{remark}

We finally turn to the case where $E$ is denumerable. Fix two disjoint
subsets $A$, $B$ of $E$ and suppose that $B^c \supset A$ is
finite. Similarly to what we did earlier in this section, we
define chain where the infinite set $B$ is collapsed to a state.

Denote by $\oX_t$ the Markov process on the finite set $B^c \cup \{\mf
d\}$, where $\mf d$ is an extra site added to $E$ to represent the
collapsed, possibly infinite, set $B$, whose rates $\Or (x,y)$, $x$, $y
\in B^c \cup \{\mf d\}$, are defined by
\begin{equation}
\label{23}
\begin{split}
& \Or(x,y) \;=\; r(x,y) \;, 
\quad \Or(x,\mf d) \;=\; \sum_{z\in B} r(x,z) \;, 
\quad x, y \in B^c\;, \\
& \quad \Or(\mf d,x) \;=\; \sum_{y\in B} \mu(y) \, 
r(y,x) \;, \quad x \in B^c \;. 
\end{split}
\end{equation}
Note that $\Or(\mf d,x)$ is finite because $\sum_{y\in E} \mu(y) \,
r(y,x)= M(x)<\infty$, as $\mu$ is a stationary state, and that
$\Or(\mf d,x) >0$ if there exists $z\in B$ such that $r(z,x)>0$.  In
particular, the collapsed chain $\{\oX_t : t\ge 0\}$ inherits the
irreducibility from the original chain. Moreover, since $\sum_{x\in
  B^c} \sum_{y\in B} \mu(y) r(y,x) = \sum_{x\in B} \sum_{y\in B^c}
\mu(y) r(y,x)$, $\omu (x) = \mu(x)$, $x\in B^c$, $\omu (\mf d) =1$ is a
stationary measure.

Let $\overline{\bb P}_x$, $x\in B^c\cup \{\mf d\}$, represent the
probability measure on the path space $D(\bb R_+, B^c\cup \{\mf d\})$
induced by the Markov process $\oX_t$ starting from $x$. Clearly, for
any $A\subset B^c$,
\begin{equation*}
\bb P_y \big [ T_{B} < T_A^+ \big] \;=\; 
\overline{\bb P}_y \big [ T_{\mf d} < T_A^+ \big] \;, \quad y\in B^c \;. 
\end{equation*}
Therefore, by \eqref{27}, for any $A\subset B^c$,
\begin{equation}
\label{24}
\Cap (A, B) \;=\; \sum_{y\in A} M(y)\, \bb P_y \big [ T_{B} < T_A^+ \big] \;=\; 
\sum_{y\in A} \oM(y)\,  \overline{\bb P}_y \big [ T_{\mf d} < T_A^+
\big] \;=\; \ocap (A, \mf d)\;,
\end{equation}
if $\ocap$ stands for the capacity of the collapsed chain.

Denote by $\oL$ the generator of the collapsed chain. Fix a pair of
functions $f$, $h: B^c \cup \{\mf d\} \to \bb R$ such that $h(\mf
d)=0$. Let $F$, $H: E\to \bb R$ be the functions defined by $F(x) =
f(x)$, $x\in B^c$, $F(z) = f(\mf d)$, $z\in B$, with a similar
definition for $H$. We claim that
\begin{equation}
\label{22}
\< \oL f, h\>_{\omu} \;=\; \< L F, H\>_{\mu} \;.
\end{equation}
Conversely, if $F$, $H: E\to \bb R$ are constant on the set $B$ and if
$H$ vanishes at $B$, \eqref{22} holds if $f$, $h: B^c \cup \{\mf d\}
\to \bb R$ are defined by $f(x) = F(x)$, $x\in B^c$, $f(\mf d) =
F(z)$, $z\in B$, with a similar definition for $h$.

To prove \eqref{22}, fix a pair of functions $f$, $h: B^c \cup \{\mf
d\} \to \bb R$ with the above properties. By definition of the
collapsed chain and since $h(\mf d)=0$, $\< \oL f, h\>_{\omu}$ is
equal to 
\begin{equation*}
\sum_{x,y\in B^c} \mu(x) \, r(x,y)\, h(x)
\, [f(y)-f(x)] \;+\; \sum_{x\in B^c} \mu(x) \, \Or(x,\mf d)\, h(x)
\, [f(\mf d)-f(x)] \;.
\end{equation*}
Since $\Or(x,\mf d) = \sum_{z\in B} r(x,z)$ and since $F$ is constant
over $B$, the second term is equal to
\begin{equation*}
\sum_{x\in B^c} \sum_{y\in B} \mu(x) \, r(x,y)\, h(x)
\, [f(\mf d)-f(x)] \;=\; \sum_{x\in B^c} \sum_{y\in B} \mu(x) \, r(x,y)\, H(x)
\, [F(y)-F(x)]\;.
\end{equation*}
Hence, adding the two terms,
\begin{equation*}
\< \oL f, h\>_{\omu} \;=\;  \sum_{x\in B^c} \sum_{y\in E} \mu(x) \, r(x,y)\, H(x)
\, [F(y)-F(x)] \;=\; \< L F, H\>_{\mu}  
\end{equation*}
because $H$ vanishes on $B$. This proves the first assertion of
claim. The converse one is proved by following the previous
computation in the reverse order.
\medskip

\noindent{\sl Proof of Theorem \ref{s10}}.  Fix two disjoint subsets
$A$, $B$ of $E$ and assume that $B^c$ is finite.  By \eqref{24}, $\Cap
(A, B) = \ocap (A, \mf d)$. On the other hand, and by Lemma \ref{s14}
and by Remark \ref{s19},
\begin{equation*}
\ocap (A, \mf d) \;=\; \inf_f \, \sup_h \Big\{ 2 \<  f \,,\, \oL
h\>_{\omu}  \, - \, \< h , (- \oL) h\>_{\omu} \Big\}\;,
\end{equation*}
where the supremum is carried over all functions $h:B^c \cup \{\mf d\}
\to \bb R$ which are constant at $A$ and vanish at $\mf d$, and where
the infimum is carried over all functions $f$ which are equal to $1$
at $A$ and $0$ at $\mf d$. Since $f$ vanishes at $\mf d$, by claim
\eqref{22}, the right hand side of the previous is equal to 
\begin{equation*}
\inf_F \, \sup_H \Big\{ 2 \<  L^* F \,,\, 
H\>_{\mu}  \, - \, \< H , (- L) H\>_{\mu} \Big\}
\end{equation*}
where the supremum is carried over all functions $H:E \to \bb R$ which
are constant at $A$ and vanish at $B$, and where the infimum is
carried over all functions $F$ which are equal to $1$ at $A$ and $0$
at $B$. The expression inside braces in the previous formula remains
unchanged if we replace $H$ by $H+c$, where $c$ is a constant. We may
therefore veil the assumption that $H$ vanishes at $B$. This proves
the first assertion of Theorem \ref{s10}.


By \eqref{24} and by Lemma \ref{s14},
\begin{equation*}
\Cap (A, B) \;=\; \ocap (A, \mf d) 
\;=\; \sup_h \Big\{ 2 \<  f_{A,\mf d} \,,\, \oL
h\>_{\omu}  \, - \, \< h , (- \oL) h\>_{\omu} \Big\}\;,
\end{equation*}
where $f_{A,\mf d} = (1/2) \{ V_{A,\mf d} + V^*_{A,\mf d}\}$, and
$V_{A,\mf d}$, $V^*_{A,\mf d}$ are the harmonic functions associated
to the collapsed process and to its adjoint. By \eqref{22}, 
\begin{equation*}
\Cap (A, B) 
\;=\; \sup_H \Big\{ 2 \<  L^* F_{A,B} \,,\, H\>_{\mu}  
\, - \, \< H , (- L) H\>_{\mu} \Big\}\;,
\end{equation*}
where $F_{A,B}(x)= f_{A,\mf d}(x)$, $x\in B^c$, $F_{A,B}(z)= 0$, $z\in
B$. By construction of the collapsed process, $V_{A,\mf d} = V_{A,B}$ and
$V^*_{A,\mf d} = V^*_{A,B}$ on $B^c$, where $V_{A,B}$ and $V^*_{A,B}$
are the harmonic functions of the original process. \qed

\medskip
\noindent{\sl Proof of Lemma \ref{s11}.} Fix two disjoint subsets $A$,
$B$ of $E$ and assume that $B^c$ is finite. By Theorem \ref{s10}, the
capacity $\Cap (A,B)$ is given by \eqref{20}. By the sector condition,
the expression inside braces in this formula is bounded by
\begin{equation*}
2 \sqrt{C_0} \<  (-S) F \,,\, F\>_{\mu}^{1/2} \<  (-S) H \,,\, H\>_{\mu}^{1/2}  
\, - \, \< H , (- S) H\>_{\mu}\;.
\end{equation*}
The supremum over $H$ is thus bounded by $C_0 \<  (-S) F \,,\,
F\>_{\mu}$. Therefore,
\begin{equation*}
\Cap (A,B) \;\le\; C_0 \inf_F \<  (-S) F \,,\, F\>_{\mu}\;,
\end{equation*}
where the infimum is carried over all functions $F$ equal to $1$ at
$A$ and $0$ at $B$. By definition of the capacity in the reversible
case, the right hand side is equal to $C_0 \Cap^s (A,B)$. This proves
the lemma in the case where the set $B^c$ is finite. To extend it to
the general case, it remains to apply Lemma \ref{s04}. \qed

\section{Flows and Proof of Theorem \ref{s02}}
\label{sec3}

We assume in this section that the state space $E$ is finite.  We
first prove Theorem \ref{s02} in the case where the sets $A$ and $B$
are singletons. The proof relies on an identity, established in Lemma
\ref{s03} below, which provides a variational formula for the norm
$\<f\,,\, \{[(-L)^{-1}]^s\}^{-1}f\>_\mu^{1/2}$.

Before stating this result, we start with an elementary
observation. We claim that
\begin{equation}
\label{39}
\text{two gradient flows $\Psi_f$, $\Psi_g$ are equal if and only 
if $f-g$ is constant}\;.
\end{equation}
Indeed, if the gradient flows are equal, since $\Psi_f- \Psi_g =
\Psi_{f-g}$, in view of \eqref{10}, $\<(-L) (f-g), (f-g)\>_\mu=0$
which implies that $f-g$ is constant. The converse is obvious.

Recall that we denote by $\mc C$ the set of divergence free flows and
by $\Phi_f$ the flow associated to a function $f:E\to \bb R$
introduced in \eqref{30}.

\begin{lemma}
\label{s03}
For every function $f:E\to \bb R$,
\begin{equation*}
\<f\,,\, \{[(-L)^{-1}]^s\}^{-1}f\>_\mu\;=\;
\< L^*f, (-S)^{-1} L^*f \>_\mu \;=\; \inf_{\varphi \in \mc C} 
\Vert \Phi_f - \varphi \Vert^2 \;.
\end{equation*}
\end{lemma}

\begin{proof}
Fix a function $f:E\to \bb R$.  Since $\Phi_f$ is a flow, by
\eqref{31} and by \eqref{39} there is a function $W:E\to\bb R$, unique
up to an additive constant, and a unique divergence free flow
$\Delta_f$ such that
\begin{equation*}
\Phi_f \;=\; \Psi_{W} \;+\; \Delta_f\;.
\end{equation*}
Computing the divergences of each flow we obtain that $L^* f = SW$ so
that $W=S^{-1}L^*f + c_0$ for some constant $c_0$. Therefore, since
$\Psi_W = \Psi_{W+c}$ for any constant $c$, $\Psi_{V}$, with $V =
S^{-1}L^*f$, is the the projection of the flow $\Phi_f$ on the space
of gradient flows. Moreover, by \eqref{10},
\begin{equation}
\label{34}
\< \Psi_V \,,\, \Psi_V \> \;=\; \< V, (-S) V \>_\mu
\;=\; \< L^*f, (-S)^{-1} L^*f \>_\mu\;,
\end{equation}
because $\< L^*f, 1 \>_\mu =0$ as $\mu$ is invariant. Furthermore,
since $\Psi_V$ is the projection of the flow $\Phi_f$ on the space of
gradient flows,
\begin{equation*}
\< \Psi_V \,,\, \Psi_V \> \;=\;
\inf_{\varphi \in \mc C} \< \Phi_f - \varphi \,,\, \Phi_f - \varphi
\>\;, 
\end{equation*}
which concludes the proof of the lemma.
\end{proof}

\begin{lemma}
\label{s09}
Fix a pair of points $a\not = b$ in $E$. Then,
\begin{equation*}
\Cap (\{a\}, \{b\}) \;=\; \inf_f \, \inf_{\varphi\in \mc C} 
\Vert \Phi_f - \varphi \Vert^2  \;,
\end{equation*}
where the infimum is carried over all functions $f:E\to \bb R$ such
that $f(a)=1$, $f(b)=0$.  Moreover, the infimum is uniquely attained
at
\begin{equation}
\label{33}
f = (1/2) \{ V_{a,b} + V^*_{a,b}\} \;, \quad
\varphi \;=\; (1/2) \, \big\{ \Phi_{V^*_{a,b}} - \Phi^*_{V_{a,b}}
\big\}\;,
\end{equation}
provided for a function $g:E\to \bb R$ we denote by $\Phi^*_g$ the
flow given by $\Phi^*_g (x,y) = g(x) c^*(x,y) - g(y) c^*(y,x)$.
\end{lemma}

\begin{proof}
The first assertion of the lemma follows from Lemmas \ref{s05} and
\ref{s03}. Moreover, the function $f$ which solves the variational
problem for the capacity coincides with the one which solves the
variational problem \eqref{07}. Hence, by Lemma \ref{s05}, $(1/2) \{
V_{a,b} + V^*_{a,b}\}$ is the unique functions which attains the minimum.
It remains to show that $(1/2) \, \{ \Phi_{V^*_{a,b}} -
\Phi^*_{V_{a,b}} \}$ is the optimal divergence free flow.

Let $F = (1/2) \{ V_{a,b} + V^*_{a,b}\}$. We claim that $(L^*F)(x) =
(SV_{a,b})(x)$ for all $x\in E$. For $x\not = a$, $b$, this identity
is obvious and has been derived in the proof of Lemma \ref{s05}. For
$x=a$, it reduces to the identity $\bb P^*_a[T^+_b < T^+_a] = \bb
P_a[T^+_b < T^+_a]$ which, in view of \eqref{27}, is equivalent to
$\Cap (\{a\},\{b\}) = \Cap^*(\{a\},\{b\})$. Since this identity is the
content of Lemma \ref{s04}, and since the same argument applies to
$x=b$, the claim is in force. In particular, by the proof of Lemma
\ref{s03}, $\Psi_{V_{a,b}}$ is the projection of the flow $\Phi_F$ on
the space of gradient flows, and there is a unique divergence free
flow $\Delta_F$ such that
\begin{equation*}
\Phi_F \;=\; \Psi_{V_{a,b}} \;+\; \Delta_F\;, \quad
\< \Phi_F - \Delta_F \,,\, \Phi_F - \Delta_F \> \;=\; 
\inf_{\varphi\in \mc C} \< \Phi_F - \varphi \,,\, \Phi_F - \varphi\>\;.
\end{equation*}
An elementary computations shows that $\Delta_F = \Phi_F -
\Psi_{V_{a,b}} = (1/2) \{ \Phi_{V^*_{a,b}} - \Phi^*_{V_{a,b}} \}$,
which completes the proof of the lemma.
\end{proof}

We may restate the previous lemma to obtain a variational formula for
the capacity in terms of the Dirichlet form.

\begin{lemma}
\label{s07}
Fix a pair of points $a\not = b$ in $E$. Then,
\begin{equation*}
\Cap (\{a\}, \{b\}) \;=\; \inf_V  D(V) \;,
\end{equation*}
where the infimum is carried over all functions $V:E\to\bb R$ such
that $\Psi_V$ is the orthogonal projection on the space of gradient
flows of some flow $\Phi_f$ with $f(a)=1$, $f(b)=0$. Moreover, the
infimum is uniquely attained, up to additive constants, at $V=V_{a,b}$.
\end{lemma}

\begin{proof}
By Lemma \ref{s05},
\begin{equation*}
\Cap (\{a\}, \{b\}) \;=\; \inf_f \< L^*f, (-S)^{-1} L^*f \>_\mu\;,
\end{equation*}
where the infimum is carried over all functions $f:E\to\bb R$ such
that $f(a)=1$, $f(b)=0$. To conclude the proof of the first assertion
of the lemma it remains to recall identity \eqref{34}.

To prove uniqueness of $V_{a,b}$, recall from the proof of Lemma
\ref{s09} that $S V_{a,b} = L^* F$, where $F= (1/2) \{ V_{a,b} +
V^*_{a,b}\}$. Hence, by the proof of Lemma \ref{s03}, $\Psi_{V_{a,b}}$
is the orthogonal projection of $\Phi_{F}$. Therefore, $D(V_{a,b}) \ge
\inf_V D(V)$. On the other hand, by \eqref{34} and since by Lemma
\ref{s05}, $F$ is the optimal function, $D(V_{a,b}) \le \inf_V
D(V)$. This shows that $V_{a,b}$ is optimal.

To prove uniqueness, suppose that $W$ is another optimal function, and
that $\Psi_W$ is the orthogonal projection on the space of gradient
flows of some flow $\Phi_g$ with $g(a)=1$, $g(b)=0$. By the optimality
of $W$ and by \eqref{34}
\begin{equation*}
\Cap (\{a\}, \{b\}) \;=\; D(W)\;=\; \< L^*g, (-S)^{-1} L^*g \>_\mu\;.
\end{equation*}
Hence, by the uniqueness of Lemma \ref{s05}, $g=F$, and by the proof
of Lemma \ref{s03}, $L^*F = SW$. Since $L^*F$ is also equal to $S
V_{a,b}$, we obtain that $S V_{a,b} = SW$, which implies that $V_{a,b}
-W$ is constant, as claimed.
\end{proof}

The flows $\Phi_{V^*_{a,b}}$ and $\Phi^*_{V_{a,b}}$ which appear in
the previous lemma have a simple probabilistic interpretation.  Denote
by $\{\bb X_n : n\ge 0\}$ the discrete time skeleton of the chain, and
recall that $M(x) = \mu(x) \lambda (x)$ is a stationary state for $\bb
X_n$, unique up to a multiplicative constant.  For $B\subset E$, let
$G_B$ be the Green function of the process killed at $B$:
\begin{equation*}
G_B(x,y) \;:=\; \bb E_x \Big[ \sum_{n=0}^{\tau_B-1} \mb 1\{\bb X_n = y\}
\Big]\;, 
\end{equation*}
where $\tau_B$ (resp. $\tau^+_B$) stands for the hitting time of
(resp. return time to) $B$ for the discrete time chain $\bb X_n$:
\begin{equation*}
\tau_B \;=\; \min\{n\ge 0 : \bb X_n \in B\}\;, \quad
\tau^+_B \;=\; \min\{n\ge 1 : \bb X_n \in B\}\;.
\end{equation*}
In the same way, $G^*_B$, $B\subset E$, stands for the Green function
of the time reversed chain killed at $B$.

Denote by $\bb P_x$, $x\in E$, the probability on path space $D(\bb
Z_+ , E)$ induced by the Markov chain $\{\bb X_n : n\ge 0\}$ starting
from $x$, and by $\theta_n$, $n\ge 0$, the time shift by $n$ units of
time.  Fix two disjoint subsets $A$, $B$ of $E$. By the last exit
decomposition, for every $x\in E$,
\begin{equation*}
\begin{split}
\bb P_x \big[ \tau_A < \tau_B \big] \; &=\;
\sum_{n\ge 0} \bb P_x \big[ \bb X_n\in A \,,\, n < \tau_B \,,\, \tau^+_B \circ
\theta_n < \tau^+_A \circ \theta_n \big] \\
&=\; \sum_{a\in A} \sum_{n\ge 0} \bb P_x \big[ \bb X_n=a \,,\, n < \tau_B
\big] \, \bb P_a \big[ \tau^+_B  < \tau^+_A \big] \\
&=\; \sum_{a\in A} G_B(x,a) \, \bb P_a \big[ \tau^+_B  < \tau^+_A
\big]\;. 
\end{split}
\end{equation*}
Since $M(x) G_B(x,y) = M(y) G^*_B(y,x)$, it follows from the previous
identity 
\begin{equation}
\label{03}
V_{A,B}(x)\;=\; \bb P_x \big[ \tau_A < \tau_B \big]  
\;=\; \sum_{a\in A} \frac 1{M(x)} \, G^*_B(a,x) \, M(a)\,
\bb P_a \big[ \tau^+_B  < \tau^+_A \big]\;.
\end{equation}

Denote by $\nu_{A,B}$ the harmonic measure, also called the normalized
charge distribution,
\begin{equation*}
\nu_{A,B} (a) \;=\; \frac 1{\Cap (A,B)}\, 
M(a)\, \bb P_a \big[ \tau^+_B  < \tau^+_A \big]\;.
\end{equation*}
Fix two disjoint subsets $A$, $B$ of $E$.  Denote by $i(x,y) =
i_{A,B}(x,y) $ the current through the arc $(x,y)$ for the process
which starts from the harmonic measure $\nu_{A,B}$ and which is killed
at $B$:
\begin{equation*}
i(x,y) \;:=\; \bb E_{\nu_{A,B}} \Big[ 
\sum_{n=0}^{\tau_B-1} \big\{ \mb 1\{\bb X_n = x, \bb X_{n+1}=y \} 
- \mb 1\{\bb X_n = y, \bb X_{n+1}=x \} \big\} \Big]\;.
\end{equation*}
By the Markov property and in view of \eqref{03}, if we denote by
$i^*(x,y)$ the current through the arc $(x,y)$ for the time reversed
chain, 
\begin{equation}
\label{38}
\begin{split}
i^*(x,y) \; & :=\; \sum_{a\in A} \nu_{A,B}(a) 
\{G^*_B(a,x) p^*(x,y) \;-\; G^*_B(a,y) p^*(y,x)\} \\
& =\; \sum_{a\in A} \nu_{A,B}(a) \Big\{ \frac 1{M(x)} G^*_B(a,x)
c^*(x,y) \;-\; \frac 1{M(y)} G^*_B(a,y) c^*(y,x) \Big\} \\
& =\; \Cap(A,B)^{-1} \big\{ V_{A,B}(x) \, c^*(x,y) \;-\;  
V_{A,B}(y)\, c^*(y,x)\big\}  \;.
\end{split}
\end{equation}

Since this last expression is equal to $\Cap(A,B)^{-1}
\Phi^*_{V_{A,B}} (x,y)$, $\Phi^*_{V_{A,B}}$ is, up to the
multiplicative constant $\Cap(A,B)$, the current through the arc
$(x,y)$ for the time reversed Markov chain $\bb X^*_n$ started from
the harmonic measure $\nu_{A,B}$ and killed at $B$. Analogously,
$\Phi_{V^*_{A,B}}$ is, up to the same multiplicative constant, the
current through the arc $(x,y)$ of the discrete time Markov chain $\bb
X_n$ started from the harmonic measure $\nu^*_{A,B}$ and killed at
$B$.

Given a function $f:E\to\bb R$, we may write the flow $\Phi_f$ as
$\Phi_f = \Psi_f + \Upsilon_f$, where $\Upsilon_f$ is the flow
given by
\begin{equation*}
\Upsilon_f (x,y) \;=\; c_a(x,y) \, \{ f(x) +  f(y) \} \;.
\end{equation*}
It turns out that the flows $\Psi_f$ and $\Upsilon_f$ are orthogonal:
\begin{equation}
\label{04}
\<\Psi_f , \Upsilon_f \>\;=\; \frac 12 \sum_{(x,y)\in \mc E} \frac 1{c_s(x,y)} \, 
\Psi_f(x,y) \, \Upsilon_f (x,y) \;=\; 0\;. 
\end{equation}
Indeed, by definition of the flows $\Psi_f$ and $\Upsilon_f$, 
\begin{equation*}
\sum_{(x,y)\in \mc E} \frac 1{c_s(x,y)} \, \Psi_f(x,y) \, \Upsilon_f (x,y)
\; =\; \sum_{x,y\in E} c_a(x,y)\, \{ f(x)^2 -  f(y)^2 \} \;.
\end{equation*}
Since $c_a(x,y) =  (1/2)\{c(x,y) - c(y,x)\} = (1/2)\{c(x,y) -
c^*(x,y)\}$, the previous expression is equal to
\begin{equation*}
\frac 12 \,  \sum_{x\in E} M(x)\, (I-P)f^2 (x)  
\;-\;  \frac 12 \, \sum_{x\in E} M(x) (I-P^*)f^2 (x) \;=\;0\;,
\end{equation*}
where $P$ represents the operator in $L^2(M)$ defined by $(Pg)(x) =
\sum_{y\in E} p(x,y) g(y)$, and where $P^*$ stands for the adjoint of
$P$ in $L^2(M)$. This proves \eqref{34}.

This orthogonality permits to restate Lemma \ref{s09} in a slightly
different form, quite useful in some cases. 

\begin{lemma}
\label{s16}
Fix a pair of points $a\not = b$ in $E$. Then,
\begin{equation*}
\Cap (\{a\}, \{b\}) \;=\; \inf_f \, \inf_{\varphi\in \mc C} 
\big \{ D(f) \;+\; \Vert \Upsilon_f - \varphi \Vert^2 \big\} \;,
\end{equation*}
where the infimum is carried over all functions $f:E\to \bb R$ such
that $f(a)=1$, $f(b)=0$.  
\end{lemma}

We are now ready to prove Theorem \ref{s02}.

\medskip
\noindent{\sl Proof of Theorem \ref{s02}.}
We proceed in two steps, collapsing each set at a time.  Fix two
disjoint subsets $A$, $B$ of $E$ and recall the notation introduced
around \eqref{18}. We first prove that
\begin{equation}
\label{41}
\inf_F \inf_\varphi \, \Vert \Phi_F - \varphi \Vert^2
\;=\; \inf_f \inf_\psi \, \Vert \Phi_f - \psi \Vert^2_A\;,
\end{equation}
where the infimum on the left hand side is carried over all functions
$F:E\to\bb R$ constant over $A$ and flows $\varphi$ such that $(\dv
\varphi)(x) =0$, $x\in A^c$, $\sum_{x\in A} (\dv \varphi)(x) =0$;
while on the right hand side $\Vert\,\cdot\,\Vert_A$ represents the
norm associated to the scalar product introduced in \eqref{42} on the
set $\oE_A$ for the process $\oX_t$ and the infimum is carried over
all functions $f:\oE_A\to\bb R$ and divergence free flows $\psi$ on
$\oE_A$.

Consider a function $F:E\to \bb R$ constant in $A$ and a flow
$\varphi$ on $E$ such that $(\dv \varphi)(x) =0$, $x\in A^c$,
$\sum_{x\in A} (\dv \varphi)(x) =0$. Recall the definition of the
function $f:\oE_A\to\bb R$ introduced below \eqref{17} and let $\psi$
be the flow on $\oE_A$ given by
\begin{equation*}
\psi (x,y) \;=\; \varphi (x,y)\;, \quad
\psi (x,\mf d) \;=\; \sum_{y\in A} \varphi (x,y)\;,
\quad x\,, y\in A^c\;.
\end{equation*}
One checks that $\psi$ is a divergence free flow. Moreover,
by Schwarz inequality,
\begin{equation*}
\Vert \Phi_f - \psi \Vert^2_A \;\le\; \Vert \Phi_F - \varphi
\Vert^2\;. 
\end{equation*}
It follows from this estimate that the left hand side of \eqref{41} is
greater than or equal to the right hand side.

Conversely, fix a function $f:\oE_A\to\bb R$ and a divergence free
flow $\psi$ on $\oE_A$. Let $F:E\to\bb R$ be the function defined
above \eqref{17}, and let $\varphi$ be the flow in $E$ given by
\begin{equation*}
\begin{split}
& \varphi (x,y) \;=\; \psi(x,y)\;, \quad x\,,y\in A^c\;,\qquad
\varphi (z,w) \;=\; 2 f(\mf d) \, c_a(z,w) \;, \quad z\,,w\in A\;, \\
& \varphi (x,y) \;=\; \Phi_F (x,y) \;-\; \frac{c_s(x,y)}{\sum_{z\in A}
  c_s(x,z)} \Big\{ \sum_{z\in A} \Phi_F(x,z) - \psi(x,\mf d)\Big\}
\;, \quad x\in A^c\,,y\in A \;.
\end{split}
\end{equation*}
One checks that $(\dv \varphi)(x)=0$, $x\in A^c$,
that $\sum_{x\in A} (\dv \varphi)(x)=0$, and that $\Vert \Phi_F -
\varphi \Vert =\Vert \Phi_f - \psi \Vert_A$. Therefore, the left hand
side of \eqref{41} is less than or equal to the right hand side,
proving claim \eqref{41}.

We are now in a position to prove the theorem.  Fix a site $x\in E$
and a set $A\not\ni x$. By \eqref{11}, $\Cap (\{x\},A) = \ocap (\{x\},
\{\mf d\})$. The assertion of the theorem when the set $B$ is a
singleton follows from Lemma \ref{s09} and \eqref{41}. The general
case is proved analogously by first collapsing the set $A$ and then
collapsing the set $B$. \qed \smallskip

We conclude this section with a bound on the capacity in the
denumerable case. Assume that $E$ is a countable set, fix a site $x\in
E$ and a set $B\not\ni x$, with $B^c$ finite. Then,
\begin{eqnarray}
\label{43}
\!\!\!\!\!\!\!\! &&
\Cap(\{x\}, B) \;\le\; \\
\!\!\!\!\!\!\!\! && \quad
\inf_F \Big\{ D(F) \;+\; 
\frac 12 \sum_{\substack{(x,y)\in \mc E \\ x,y\in B^c}}
\frac{c_a(x,y)^2}{c_s(x,y)} \{ \, F (x) + F (y) - 2 \,
\}^2 \;+\; 4 \sum_{\substack{(x,z)\in \mc E \\ x\in B^c, z\in B}}
\frac{c_a(x,z)^2}{c_s(x,z)}  \Big\},
\nonumber
\end{eqnarray}
where the infimum is carried over all functions $F:E\to\bb R$ such
that $F(x)=1$, $F(z)=0$, $z\in B$.

Indeed, recall the notation introduced around \eqref{23}. Denote by
$\oE$ the set $B^c \cup \{\mf d\}$, where $\mf d$ is an extra site added
to $E$. Let $\{\oX_t : t\ge 0\}$ be the process obtained from $X_t$ by
collapsing the set $B$ to the point $\mf d$, and let $\oD$, $\omE$,
$\oc (x,y)$ and $\ocap$ be the associated Dirichlet form, oriented
bonds, conductances and capacities, respectively.

By \eqref{24}, $\Cap(\{x\}, B) =\ocap (\{x\}, \{\mf d\})$. Fix a
function $F:E\to \bb R$ which vanishes on $B$ and is equal to $1$ at
$x$, and let $f:\oE \to \bb R$ be given by $f(y)=F(y)$, $y\in B^c$,
$f(\mf d)=0$. Since $\oc_a(x,y)$ is a divergence free flow, by Lemma
\ref{s16},
\begin{equation*}
\ocap (0,\mf d) \;\le\; \oD (f) \;+\; \frac 12 \sum_{(x,y)\in \omE}
\frac 1{\oc_s (x,y)}  \big\{ \oc_a(x,y) [f(x) + f(y)] - 2
\oc_a(x,y)\}^2\; . 
\end{equation*}
Clearly, $\oD (f) = D(F)$. On the other hand, by \eqref{23} if the arc
$(x,y)$ is contained in $B^c$, we may replace $\oc_s(x,y)$,
$\oc_a(x,y)$ and $f$ by $c_s(x,y)$, $c_a(x,y)$ and $F$,
respectively. In contrast, for an arc $(x,\mf d)$, $x\in B^c$, since
$\oc_t (x,\mf d) = \sum_{z\in B} c_t (x,z)$, $t=a,s$, and since $f
(\mf d)=0=F (z)$, $z\in B$, by Schwarz inequality,
\begin{equation*}
\begin{split}
& \Big\{ \oc_a(x,\mf d) \big[\, f (x) + f (\mf d) - 2 \, \big] \Big\}^2
\;=\; \Big\{ \sum_{z\in B} c_a(x,z) [ \, F (x) + F (z) -
2\, ] \Big \}^2 \\
&\quad \le\;
\sum_{z\in B} \frac{c_a(x,z)^2}{c_s(x,z)} [ F (x) + F (z) -
2]^2 \, \sum_{z\in B} c_s(x,z)\;.
\end{split}
\end{equation*}
Since $0\le F\le 1$, $F (x) + F (z) - 2$ is absolutely bounded by
$2$. Putting together all previous estimates we derive \eqref{43}.

\section{Recurrence criteria}
\label{sec4}

It is well known that in the reversible case the Dirichlet and the
Thomson principle provide powerful tools to prove the recurrence or
the transience of irreducible Markov processes evolving in countable
state spaces. In this section, we examine this matter in the non
reversible case.

Consider a irreducible Markov process $\{X_t : t\ge 0\}$ on a
countable state space $E$ satisfying the assumptions of the beginning
of Section \ref{sec1}. We assume, in particular, the existence of a
stationary state $\mu$. 

It is well known that the Markov process $X_t$ is recurrent if and
only if there exist a site $0\in E$ and a sequence of \emph{finite}
subsets $B_n$ containing $0$ and increasing to $E$, $B_n\subset
B_{n+1}$, $\cup_n B_n = E$, such that
\begin{equation*}
\lim_{n\to\infty} \bb P_0 \big [ T_{B_n^c} < T_0^+ \big]\;=\;0\;.
\end{equation*}
By \eqref{27}, for any finite set $B$ containing the site $0$,
\begin{equation*}
\frac 1{M (0)} \, \bb P_0 \big [ T_{B^c} < T_0^+ \big] \;=\; 
\Cap (0, B^c)\;.
\end{equation*}
Hence, the Markov process $X_t$ is recurrent if and only if there
exist a site $0\in E$ and a sequence of finite subsets $B_n$
containing $0$ and increasing to $E$ such that
\begin{equation}
\label{26}
\lim_{n\to\infty} \Cap (0, B_n^c) \;=\; 0\;.
\end{equation}
The proof of the recurrence is thus reduced to the estimation of the
capacity between a site and the complement of a finite set. This
problem can be further simplified by collapsing the set $B^c_n$ to a
point, as we did in Section \ref{sec2}.

The first two results follow from the previous observation and the
bounds stated in Lemmas \ref{s15} and \ref{s11}. Recall that $\{X^s_t
\,|\, t\ge 0\}$ stands for the reversible version of the process $X_t$
whose generator is given by $S$.

\begin{lemma}
\label{s17}
Let $\{X_t \,|\, t\ge 0\}$ be a irreducible Markov process on a
countable state space $E$ which admits a stationary measure. The
process is transient if so is the Markov process $\{X^s_t \,|\, t\ge
0\}$.
\end{lemma}

\begin{lemma}
\label{s13}
Let $\{X_t \,|\, t\ge 0\}$ be a irreducible Markov process on a
countable state space $E$ which admits a stationary measure. The
process is recurrent if its generator satisfies a sector condition and
if the Markov process $\{X^s_t \,|\, t\ge 0\}$ is recurrent.
\end{lemma}

Cycle random walks with bounded cycles, \cite{ma}, \cite{klo}, mean
zero asymmetric exclusion process \cite{v}, or asymmetric zero range
process on a finite cylinder \cite{gl3} are examples of non reversible
Markov processes which satisfy the sector condition.

\begin{lemma}
\label{s12}
Let $\{X_t \,|\, t\ge 0\}$ be a irreducible Markov process on a countable
state space $E$ which admits a stationary measure. The process is
recurrent if the Markov process $\{X^s_t \,|\, t\ge 0\}$ is recurrent
and if
\begin{equation*}
\sum_{(x,y)\in \mc E} \frac{c_a(x,y)^2}{c_s(x,y)} \;<\; \infty\;.
\end{equation*}
\end{lemma}

\begin{proof}
Fix $\epsilon >0$ and a site $0\in E$. By assumption, there exists a
finite set $A\ni 0$ such that
\begin{equation*}
\sum_{\substack{(x,y)\in \mc E \\ \{x,y\} \not\subset A}}
\frac{c_a(x,y)^2}{c_s(x,y)} \;\le\; \epsilon\;.
\end{equation*}
By \eqref{27}, for all subsets $B$ of $E$ such that $A\subset B^c$,
$B^c$ finite, $\Cap^s(A,B)\le\sum_{x\in A} \Cap^s(\{x\},B)$. Hence,
since the process $X^s_t$ is recurrent, by \eqref{26} and by Lemma
\ref{s01}, there exists a finite set $B^c\supset A$ such that
$\Cap^s(A,B)\le \epsilon$.

Denote by $V^s_{A,B}: E\to \bb R$ the equilibrium potential associated
to the reversible process $X^s$: $V^s_{A,B} (x) = \bb P^s_x [T_A <
T_{B}]$. Since $D(V^s_{A,B}) = \Cap^s(A,B)$, by construction of $B$,
$D(V^s_{A,B})\le \epsilon$. Therefore, by \eqref{43} with
$F=V^s_{A,B}$, $\Cap(0,B)$ is bounded above by
\begin{equation*}
\epsilon \;+\; \frac 12 \sum_{\substack{(x,y)\in \mc E \\ x,y\in B^c}}
\frac{c_a(x,y)^2}{c_s(x,y)} \{ \, V^s_{A,B} (x) + V^s_{A,B} (y) - 2 \,
\}^2
\;+\; 4 \sum_{\substack{(x,z)\in \mc E \\ x\in B^c, z\in B}}
\frac{c_a(x,z)^2}{c_s(x,z)}  \;\cdot
\end{equation*}
Since $V^s_{A,B}$ is identically equal to $1$ on $A$, the
previous expression is less than or equal to
\begin{equation*}
\epsilon \;+\; 4 \sum_{\substack{(x,y)\in \mc E \\ x\in A,y\in A^c}}
\frac{c_a(x,y)^2}{c_s(x,y)} 
\;+\; 2 \sum_{\substack{(x,y)\in \mc E \\ x,y\in A^c}}
\frac{c_a(x,y)^2}{c_s(x,z)}  \;\cdot
\end{equation*}
By definition of the set $A$, this expression is bounded by
$5\epsilon$, which concludes the proof of the lemma.
\end{proof}

\subsection{Random walks with self-similar rescaled invariant
  potential}
In \cite{du}, Durrett built from a random potential, with a large
scale self-similarity property, a reversible nearest-neighbor random
walk on $\bb Z^d$ for which Sinai random walk is a special case when
$d=1$. He proved that such a random walk is recurrent under simple and
natural assumptions on the scaling limit of the potential. Note,
however, that such a random walk could never be a particular case of
classical randoms walks in random environment in dimension $d\ge 2$
due to the reversibility condition.

We want to point out here that the key feature of this model is {\sl
  not} the reversibility but, as Durrett suggested, the existence of a
strongly fluctuating invariant measure. The reason for the restriction
to the reversible was that it allowed the use of the Dirichlet
principle. Our extended Dirichlet principle permits to reproduce
Durrett's argument with only assumptions on the invariant measure and
without the reversibility hypothesis.

Consider a discrete time, nearest-neighbor random walk $\{X_n \,|\,
n\ge 1\}$ on $\bb Z^d$ with random transition probabilities
$p(x,y):\Omega \to [0,1]$, and assume the existence of a (random)
invariant measure $\mu$. We define the invariant potential $V:\bb
Z^d\to\bb R$ by
\begin{equation*}
\mu(x) \;=\; e^{-V(x)}\; \quad x\in\bb Z^d \;,
\end{equation*}
and assume, without loss of generality, that $V(0)=0$. The relation
between the random potential and the invariant measure in \cite{du} is
not exactly the same, but this definition makes our point more
clear. We can extend $V$ into a continuous function $V:\bb R^d\to\bb
R$ and see it as a random variable in $C(\bb R^d, \bb R)$, the space
of continuous functions from $\bb R^d$ to $\bb R$ equipped with the
topology of uniform convergence on compact sets. We assume that there
exists $\alpha>0$ and a random variable $W:\Omega\to C(\bb R^d, \bb
R)$ such that $\lambda^{-\alpha} V(\lambda \,\cdot\,)$ converges in
law to $W(\,\cdot\,)$ when $\lambda\uparrow\infty$. Hence, $W$ is a
self-similar $d$-dimensional process and we have the following result.

\begin{lemma}
\label{s20}
If there is almost surely $a>0$ such that the connected component of
the origin in $\{x\in\bb R^d: W(x)<a\}$ is bounded, then $X$ is almost
surely recurrent.
\end{lemma}

We refer to \cite{du} for examples of (reversible) processes which
satisfy such hypotheses. Even though we could relax Durrett's
reversibility hypothesis, it is not clear how to build non artificial
irreversible examples in which one has enough control on an invariant
measure to check the assumptions on $V$. One can start, for example, as
in \cite{du}, with a random potential $V$ with stationary increments,
and build the reversible random walk inside this potential to have
$\mu$ as invariant measure. We may then add some irreversibility by
superposing to this reversible dynamics some drift along cycles on the
level sets of $V$, keeping $\mu$ as an invariant measure.

\medskip
\noindent{\sl Proof of Lemma \ref{s20}.}
We follow closely Durrett's proof. First, following Skorohod
\cite{sk}, we can build on the same probability space random variables
$V_1, V_2, \dots$ with the same law as $V$ and such that $n^{-\alpha}
V_n(n\,\cdot\,)$ converges {\sl almost surely} in $C(\bb R^d, \bb R)$
to $W(\,\cdot\,)$. Define $C_a$, $a>0$, as the bounded connected
component of the origin in $\{x\in\bb R^d : W(x)<a\}$, and set
\begin{equation*}
G_n = (n C_a) \cap \bb Z^d\;, \quad
\partial_- G_n = \{x\in G_n : \exists y\not\in G_n\;, \Vert
x-y\Vert=1\}\; ,
\end{equation*}
where $\Vert \,\cdot\,\Vert$ stands for the Euclidean norm. 

We claim that $\mu_n(\partial_- G_n)$ converges almost surely to $0$,
where $\mu_n(x) = e^{-V_n(x)}$, $x\in\bb Z^d$. Indeed, by assumption
there are $r$, $R>0$ such that
\begin{equation*}
C_a \subset [-R,R]^d \;,\quad W(y)\ge a/2
\end{equation*}
for all $y\in B(z,r)$, $z\in \partial C_a$, where $B(x,r)$ stands for
the ball centered at $x$ with radius $r$ and $\partial C_a$ for the
boundary of $C_a$. Therefore, almost surely, for $n$ large enough,
\begin{equation*}
\mu_n(\partial_- G_n) \;=\; \sum_{x\in \partial_- G_n} e^{-V_n(x)} \;\le\;
|G_n|\, e^{-(a/4)n^\alpha} \;\le\; n^d (2R)^d e^{-(a/4)n^\alpha}\;,
\end{equation*}
which proves the claim.

For any finite subset $B$ of $\bb Z^d$ which contains the origin,
since $\mu(0) = e^{-V(0)}=1$, we have
\begin{equation*}
\bb P_0 [ \tau^+_0 = \infty] \;\le\; \mu(0) 
\bb P_0 [ \tau^+_0 > \tau^+_{B^c}] \;=\; \Cap(0,B^c)\;.
\end{equation*}
By taking the test function $\mb 1\{B\}$ in \eqref{43} we obtain that
$\Cap(0,B^c)$ is bounded above by
\begin{equation*}
\sum_{x\in B, y\in B^c} c_s(x,y) \;+\;
4 \sum_{x\in B, y\in B^c} \frac{c_a(x,y)^2}{c_s(x,y)}
\;\le\; 5 \sum_{x\in B, y\in B^c} c_s(x,y)  \;\le\; 
5 \mu(\partial_- B)\;.
\end{equation*}
Since $\mu(0) = 1$ we also have $\Cap(0,B^c)\leq 1$,
thus $\Cap(0,B^c)\leq 5\mu(\partial_-B)\wedge 1$.
Now, for any $k > 0$,
\begin{eqnarray*}
P_0(\tau_0^+ = +\infty) 
&\leq& \min \left\{\Cap(0, B^c) \geq 0 :\: 0 \in B\subset [-k, k]^d\right\} \\
&\leq& \min_{0\in B\subset [-k,k]^d} 5\mu(\partial_-B)\wedge 1.
\end{eqnarray*}
Since $V_n$ has the same distribution as $V$, taking expected values
with respect to the environment, denoted by $\mb E$, by the monotone
convergence theorem, we obtain that, for all $n\geq 1$,
\begin{eqnarray*}
&& \mb E \big[ \,  \bb P_0[\tau_0^+ = +\infty] \, \big ]
\;\le\; \lim_{k\rightarrow \infty} 
\mb E\Big[ \min_{0\in B\subset [-k,k]^d} 5\mu(\partial_-B)\wedge 1 \Big]\\ 
&& \quad =\; \lim_{k\rightarrow \infty} 
\mb E\Big[ \min_{0\in B\subset [-k,k]^d} 5\mu_n(\partial_-B)\wedge 1
\Big] \\
&& \qquad =\; \mb E\Big[ \lim_{k\rightarrow \infty} 
\min_{0\in B\subset [-k,k]^d} 5\mu_n(\partial_-B)\wedge 1 \Big] \;\le\;
\mb E\Big[ 5\mu_n(\partial_-G_n)\wedge 1 \Big]\;.
\end{eqnarray*}
Thus, by the dominated convergence theorem,
\begin{eqnarray*}
\mb E \big[ \, \bb P_0[\tau_0^+ = +\infty] \,\big]
&\leq& \lim_{n\rightarrow +\infty} 
\mb E\big[ \, 5\mu_n(\partial_-G_n)\wedge 1 \,\big] \\
& = & \mb E\Big[ \lim_{n\rightarrow +\infty}  5\mu_n(\partial_-G_n)
\wedge 1 \Big] \;=\; 0\;,
\end{eqnarray*}
which proves that the process is almost surely recurrent. \qed \smallskip

\subsection{Two dimensional random walk in asymmetric random
  conductances.} 

The most natural way to generalize the classical random conductance
model on a graph may be the following. To define the asymmetric
conductances $c(x,y)$ on each arc $(x,y)$ we superpose symmetric
functions $c_s(x,y)$ and a divergence free flow $c_a(x,y)$ with the
restriction that $|c_a|\le c_s$ to end with nonnegative conductances
$c(x,y)$. 

More precisely, consider a family $\Gamma$ of finite cycles $\gamma$
on a countable graph $(E, \mc E)$, and a family of nonnegative
random variables $Z_\gamma$, $\gamma\in \Gamma$, such that for each
$(x,y)\in\mc E$,
\begin{equation*}
\sum_{\gamma\in\Gamma} Z_\gamma \, |\chi_\gamma(x,y)|\;<\; \infty\;,
\end{equation*}
where $\chi_\gamma$ is the divergence free flow introduced in
\eqref{40}.  Define the divergence free flow $c_a$ by
\begin{equation*}
c_a(x,y) \;=\; \sum_{\gamma\in\Gamma} Z_\gamma \, \chi_\gamma(x,y)\;,
\quad (x,y)\in \mc E\;.
\end{equation*}
There are two natural ways to define symmetric conductances in this
context. Consider a family of nonnegative random variables
$\{Y_{(x,y)} : (x,y)\in\mc E\}$ such that $Y_{(x,y)}=Y_{(y,x)}$. We
may set $c_s(x,y)= Y_{(x,y)} + |c_a(x,y)|$, or $c_s(x,y)= Y_{(x,y)} +
\sum_{\gamma\in\Gamma} Z_\gamma \, |\chi_\gamma(x,y)|$.  \smallskip

In the special case of the two dimensional lattice, we can decompose
each flow associated to a finite cycle as a linear combination of
elementary flows associated to cycles of length $4$. For $x\in\bb
Z^2$, denote by $\gamma_x$ the cycle $(x, x+e_1, x+e_1+e_2, x+e_2,x)$,
where $e_1$, $e_2$ stands for the canonical basis of $\bb R^2$. A
flow $\chi_\gamma$ associated to a finite cycle $\gamma$ can be
written as
\begin{equation*}
\chi_\gamma \;=\; \sum_{x\in\bb Z^d} W_{\gamma,\gamma_x} \, \chi_{\gamma_x} \;,
\end{equation*}
where $W_{\gamma,\gamma_x}=1$ (resp. $-1$) if the cycle $\gamma_x$ is
contained in the interior of $\gamma$ and the cycle $\gamma$ runs
counter-clockwise (resp. clockwise), and $W_{\gamma,\gamma_x}=0$ if
the cycle $\gamma_x$ is not contained in the interior of $\gamma$.

Denote by $\mb E$ expectation with respect to the random variables
$Z_\gamma$ and assume that
\begin{equation*}
\sum_{\gamma\in \Gamma} \mb E[Z_\gamma] \, |W_{\gamma,\gamma_x}| \;<\; \infty
\quad\text{for all $x\in\bb Z^d$}\;.
\end{equation*}
In this case $W_{\gamma_x} := \sum_{\gamma\in \Gamma} Z_\gamma \,
W_{\gamma,\gamma_x}$ is almost surely well defined for all $x\in\bb
Z^d$ and so is the divergence free flow $c_a$ given by
\begin{equation}
\label{35}
c_a \;=\; \sum_{x\in\bb Z^d} W_{\gamma_x}\, \chi_{\gamma_x}\;.
\end{equation}

Note that each arc $(x,y)$ belongs to exactly two elementary cycles,
denoted by $\gamma^\pm(x,y)$ and characterized by the fact that
$\chi_{\gamma^\pm(x,y)} (x,y) = \pm 1$. With this notation, for any
arc $(x,y)$, $c_a(x,y) = W_{\gamma^+(x,y)} - W_{\gamma^-(x,y)}$.

\begin{lemma}
\label{s18}
Suppose that
\begin{equation*}
\sup_{(x,y)} \mb E\Big[ \, c_s(x,y) \;+\; \frac{ [W_{\gamma^+(x,y)}]^2+
[W_{\gamma^-(x,y)}]^2}{c_s(x,y)} \, \Big] \;<\; \infty\;,
\end{equation*}
where the supremum is carried over all arcs. Then, the random walk is
almost surely recurrent.
\end{lemma}

\begin{proof}
Let $B^c_n = [-n,n]^2$, $n\ge 1$, consider a function $f_n: \bb
Z^2\to\bb R$ such that $f_n(0)=1$, $f_n(x)=0$ for $x\in B_n$, and a
divergence free flow $\psi_n = \sum_{x} a_x \chi_{\gamma_x}$, where
the sum is performed over all $x\in \bb Z^2$ for which the elementary
cycle $\gamma_x$ is contained in $B^c_n$. Repeating the proof of
\eqref{43} and keeping in mind that $c_a$ is absolutely bounded by
$c_s$, we obtain that
\begin{equation*}
\Cap(0,B_n) \;\le\; D(f_n) \;+\; \frac 12 \sum_{x,y\in B^c_n} \frac
1{c_s(x,y)} \Big\{ c_a(x,y)[f_n(x) + f_n(y)] - \psi_n(x,y)\Big\}^2\;. 
\end{equation*}

Consider the divergence free flow $\varphi_n$ given by
\begin{equation*}
\varphi_n \;=\; \frac 12 \sum_{x} F_n (\gamma_x) \, 
W_{\gamma_x} \, \chi_{\gamma_x}\;, \quad\text{where}\quad   
F_n (\gamma_x) \;=\; \sum_{z\in  \gamma_x} f_n(z) \;,
\end{equation*}
and where the sum is carried over all sites $x$ in $\bb Z^2$ for which
the elementary cycle $\gamma_x$ is contained in $B^c_n$.  By the
previous bound,
\begin{equation*}
\Cap(0,B_n) \;\le\; D(f_n) \;+\; \frac 12 \sum_{x,y\in B^c_n} \frac
1{c_s(x,y)} \Big\{ c_a(x,y)[f_n(x) + f_n(y)] -
\varphi_n(x,y)\Big\}^2\;. 
\end{equation*}

As we know, $D(f_n) = (1/2) \sum_{x,y\in\bb Z^2} c_s(x,y)
[f_n(y)-f_n(x)]^2$. On the other hand, it follows from the definitions
of the asymmetric conductance and the divergence free flow $\varphi_n$
that $c_a(x,y)[f_n(x) + f_n(y)] - \varphi_n(x,y)$ is equal to
$W_{\gamma^+(x,y)}\{f_n(x) + f_n(y) - (1/2) F_n (\gamma^+(x,y))\} -
W_{\gamma^-(x,y)} \{f_n(x) + f_n(y) - (1/2) F_n (\gamma^-(x,y))\}$ if
the arc $(x,y)$ does not belong to one side of the square $B^c_n$. The
absolute value of this difference is bounded above by
$|W_{\gamma^+(x,y)}| \max_{e\in \gamma^+(x,y)} |f_n(e^+)-f_n(e^-)| +
|W_{\gamma^-(x,y)}| \max_{e\in \gamma^-(x,y)} |f_n(e^+)-f_n(e^-)|$.
If the arc $(x,y)$ belongs to one side of the square $B^c_n$, taking
advantage of the fact that $f_n$ vanishes outside $B^c_n$, we obtain a
similar formula with an extra factor $2$.  In conclusion,
$\Cap(0,B_n)$ is bounded above by
\begin{equation*}
4\, \sum_{x,y\in\bb Z^2} \Big\{ c_s(x,y) 
+ \frac{[W_{\gamma^-(x,y)}]^2 + [W_{\gamma^+(x,y)}]^2}{c_s(x,y)} \Big\} 
\max_{e} [f_n(e^+)-f_n(e^-)]^2 \;,
\end{equation*}
where the maximum is carried over all arcs $e$ in $\gamma^-(x,y) \cup
\gamma^+(x,y)$.

Let 
\begin{equation*}
f_n(x) \;=\; \Big( 1 - \frac{\log (1+\Vert x\Vert_\infty)}{\log (n+2)}
\Big) \mb 1 \{[-n,n]^2\} (x)\;,
\end{equation*}
where $\Vert x\Vert_\infty = \max \{|x_1|,|x_2|\}$, $x=(x_1,x_2)$.
It follows from this choice and from the assumption of the lemma
that 
\begin{equation*}
\lim_{n\to\infty} \mb E[\Cap (0,B_n)]\;=\;0\;.
\end{equation*}
In particular, there exists almost surely a subsequence $(B_{n_k} :
k\ge 1)$ such that $\lim_{k\to\infty} \Cap (0,B_{n_k})=0$ and the
almost sure recurrence follows. 
\end{proof}

We conclude with an example which satisfies the assumptions of the
previous lemma. Suppose that the random variables $Z_\gamma$ are
independent Poisson variables with parameter $\lambda^{|\gamma|}$,
$0<\lambda<1/3$, and that the random variables $Y_{(x,y)}$ have a
common distribution bounded away from $0$ and with a finite first
moment. Let $c_a$ be given by \eqref{35} and let $c_s(x,y) = Y_{(x,y)}
+ |c_a(x,y)|$. We claim that the hypotheses of the previous result are
fulfilled.

Indeed, by assumption there exists $\delta>0$ such that $c_s(x,y) \ge
Y_{(x,y)}\ge \delta>0$ almost surely. Therefore, to show that the
assumptions of the previous lemma are in force we need only to prove
that
\begin{equation}
\label{36}
\sup_{(x,y)} \mb E\big[ c_s(x,y) \big] \;<\;\infty \quad
\text{and}\quad \sup_{x\in\bb Z^d} 
\mb E\big[ W_{\gamma_x}^2 \big] \;<\;\infty\;. 
\end{equation}

Since $|c_a(x,y)| \le W_{\gamma^+(x,y)} + W_{\gamma^-(x,y)}$, for
every arc $(x,y)$,
\begin{equation*}
\mb E\big[ c_s(x,y) \big] \;\le\; \mb E\big[ Y(x,y) \big] \;+\;
\sum_{p=\pm} \sum_{\gamma\in \Gamma} W_{\gamma, \gamma^p(x,y)} \,
\mb E\big[Z_\gamma \big]\;.
\end{equation*}
By assumption, the first term on the right hand side is bounded
uniformly over $(x,y)$, while the second term is less than or equal to
$\sum_{k\ge 4} 8 k 3^{k} \lambda^{k}$ because there are at most
$4\cdot 3^{k-1}$ self-avoiding walks of length $k$ and because a cycle
of length $k$ containing in its interior an elementary cycle must
cross a line parallel to one of the axis in at most $2k$ points. This
proves the first bound in \eqref{36}.  To prove the second bound, fix
an elementary cycle $\gamma_x$. By definition of the random variables
$Z_\gamma$,
\begin{equation*}
\mb E\big[ W_{\gamma_x}^2 \big] \;=\; \mb E\big[ W_{\gamma_x} \big]^2 \;+\;
\sum_{\gamma\in \Gamma} \lambda^{|\gamma|} \, W_{\gamma, \gamma_x}^2\;.
\end{equation*}
The first expectation has been estimated above, while the second one
can be estimated in the same way. This concludes the proof of
\eqref{36}.

\smallskip
\noindent{\bf Acknowledgment}. This collaboration took place during the
visit of the first author to the Unit\'e Mixte Internationale CNRS –
IMPA (UMI 2924).

\end{document}